\documentclass[11pt]{article}
\usepackage {amssymb} 
\usepackage {amsmath}
\usepackage{mathrsfs}
\usepackage{amsthm}
\usepackage{graphicx}
\usepackage {amscd}
\usepackage{subfigure}
\usepackage{stmaryrd}
\usepackage{tikz-cd}
\usepackage[new]{old-arrows}
\usepackage{xypic}

\usepackage{color}
\usepackage{multirow}
\usepackage{tabu}
\usepackage[titletoc, title]{appendix}
\usepackage{diagbox}
\usepackage[colorlinks, citecolor=red]{hyperref}

\usepackage{geometry}  

\newcommand{\vol}{{\rm vol}}
\newcommand{\ord}{{\rm ord}}

\newcommand{\fa}{\mathfrak{a}}
\newcommand{\cO}{\mathcal{O}}
\newcommand{\bR}{\mathbb{R}}
\newcommand{\bC}{\mathbb{C}}
\newcommand{\bZ}{\mathbb{Z}}

\newcommand{\Val}{{\rm Val}}

\newcommand{\hvol}{{\widehat{\rm vol}}}

\newcommand{\bQ}{{\mathbb{Q}}}

\newcommand{\bP}{{\mathbb{P}}}

\newcommand{\cC}{{\mathcal{C}}}

\newcommand{\cA}{{\mathcal{A}}}
\newcommand{\cB}{{\mathcal{B}}}

\newcommand{\scr}{\mathscr}
\newcommand{\Hom}{{\rm Hom}}
\newcommand{\CY}{{\rm CY}}
\newcommand{\id}{{\rm id}}

\newcommand{\sddb}{{\sqrt{-1}\partial\bar{\partial}}}

\newcommand{\vphi}{{\varphi}}

\newcommand{\orb}{{\rm orb}}

\newcommand{\cE}{{\mathcal{E}}}

\newcommand{\tr}{\mathrm{tr}}

\newcommand{\bN}{{\mathbb{N}}}

\newcommand{\til}{\tilde}
\newcommand{\rk}{{\rm rk}}

\newtheorem{thm}{Theorem}[section]
\newtheorem{dfpr}[thm]{Definition-Proposition}
\newtheorem{lem}[thm]{Lemma}
\newtheorem{def-prop}[thm]{Definition-Proposition}
\newtheorem{cor}[thm]{Corollary}
\newtheorem{defn}[thm]{Definition}

\newtheorem{prop}[thm]{Proposition}

\newtheorem{conj}[thm]{Conjecture}
\newtheorem{rem}[thm]{Remark}
\newtheorem{exmp}[thm]{Example}
\newtheorem{nota}[thm]{Notations}

\begin{document}

\title{On the stability of extensions of tangent sheaves on K\"{a}hler-Einstein Fano / Calabi-Yau pairs}

\author{Chi Li}

\date{}

\maketitle

\begin{center}
{\it \footnotesize To Professor Gang Tian on the occasion of his sixtieth birthday}
\end{center}

\abstract{
Let $S$ be a smooth projective variety and $\Delta$ a simple normal crossing $\bQ$-divisor with coefficients in $(0,1]$. For any ample $\bQ$-line bundle $L$ over $S$,
we denote by $\scr{E}(L)$ the extension sheaf of the orbifold tangent sheaf $T_S(-\log(\Delta))$ by the structure sheaf $\cO_S$ with the extension class $c_1(L)$. We prove the following two results:
\begin{enumerate}
\item[(1)] if $-(K_S+\Delta)$ is ample and $(S, \Delta)$ is K-semistable, then for any $\lambda\in \bQ_{>0}$, the extension sheaf $\scr{E}({\lambda c_1(-(K_S+\Delta))})$ is slope semistable with respect to $-(K_S+\Delta)$;
\item[(2)] if $K_S+\Delta\equiv 0$, then for any ample $\bQ$-line bundle $L$ over $S$, $\scr{E}(L)$ is slope semistable with respect to $L$.
\end{enumerate}
These results generalize Tian's result where $-K_S$ is ample and $\Delta=\emptyset$. We give two applications of these results. The first is to study a question by Borbon-Spotti about the relationship between local Euler numbers and normalized volumes of log canonical surface singularities. We prove that the two invariants differ only by a factor $4$ when the log canonical pair is an orbifold cone over a marked Riemann surface. In particular we complete the computation of Langer's local Euler numbers for any line arrangements in $\bC^2$. The second application is to derive Miyaoka-Yau-type inequalities on K-semistable log-smooth Fano pairs and Calabi-Yau pairs, which generalize some Chern-number inequalities proved by Song-Wang.
}
\tableofcontents

\section{Introduction}

The Hitchin-Kobayashi correspondence states that a holomorphic vector bundle over a K\"{a}hler manifold admits a Hermitian-Einstein metric if and only if it is slope polystable. This was known by the works of Narashimhan-Seshadri, Donaldson and Uhlenbeck-Yau. Correspondingly, the Yau-Tian-Donaldson conjecture for Fano manifolds states that a smooth Fano manifold admits a K\"{a}hler-Einstein metric if and only if it is K-polystable. Due to many people's work, the latter conjecture has been solved (see \cite{Tia97, Ber15, CDS15, Tia15}). A K\"{a}hler-Einstein metric is naturally a Hermitian-Einstein metric on the tangent bundle. So if a Fano manifold admits a K\"{a}hler-Einstein metric then its tangent bundle is slope polystable. 
In \cite{Tia92}, Tian discovered a deeper phenomenon that the stability or the instability of the some natural extension sheaf of the tangent sheaf can be used to bound the maximal possible positive lower bound of the Ricci curvature of K\"{a}hler metrics in $2\pi c_1(S)$. In particular he proved that

\begin{thm}[\cite{Tia92}]\label{thm-Tian}
\begin{enumerate}
\item
If a Fano manifold $S$ admits a K\"{a}hler-Einstein metric, then there is a natural Hermitian-Einstein metric $h_E$ on the extension bundle $E$ of $TS$ by the trivial line bundle with the extension class $c_1(S)$. In particular, $E$ is slope polystable.  
\item
If for any $t\in (0,1]$, there exists a K\"{a}hler metric in $2\pi c_1(S)$ with $Ric(\omega)\ge t\omega$. Then $E$ is slope semistable.
\end{enumerate}
\end{thm}
Tian's construction will be reviewed in section \ref{sec-logTian}. In this note, we will first generalize Theorem \ref{thm-Tian} to the logarithmic setting. To state the result, we first recall the following standard definition.
\begin{defn} 
Let $S$ be a
normal projective variety and $\Delta=\sum_i \delta_i \Delta_i$ be a $\bQ$-divisor with $\delta_i\in (0,1]$. We assume $(S, \Delta)$ has log canonical singularities. 
\begin{enumerate}
\item
$(S, \Delta)$ is a log-Fano pair if $-(K_S+\Delta)$ is an ample $\bQ$-Cartier divisor and $(S, \Delta)$ has klt singularities.
\item
$(S, \Delta)$ is a log-Calabi-Yau pair (log-CY) if $(K_S+\Delta)\equiv 0$.
\end{enumerate}
\end{defn}

We can state our first result and refer to \cite{Tia97, Don02, Li12} for the definition of  K-(semi)stability of log-Fano pairs.
\begin{thm}\label{thm-main}
Let $(S, \Delta)$ be a log-Fano pair that is log smooth (i.e. $S$ is smooth and $\Delta$ is simple normal crossing).
If $(S, \Delta)$ is K-semistable, then the orbifold tangent sheaf $T_S(-\log\Delta)$ is slope semistable with respect to $-(K_S+\Delta)$. 

Moreover, for any $\lambda\in \bQ_{>0}$, if we let $\scr{E}$ be the extension sheaf of $T_S(-\log \Delta)$ by $\cO_S$ with the extension class $\lambda \cdot c_1(-(K_S+\Delta))$, then $\scr{E}$ is slope semistable with respect to $-(K_S+\Delta)$.
\end{thm}
The first statement could be seen as the log-Fano correspondent of the semistability results in \cite{Gue16}, \cite[Theorem C]{GT16}. The techniques used in its proof is partly inspired by \cite{CP16, GT16}.

We can prove a stronger result by weakening the log-smooth assumption (see section \ref{sec-singularstable}). We expect that a similar statement is true by just assuming that the pair $(S, \Delta)$ is K-semistable. 
Our proof uses continuity method as in \cite{Don12, LTW17, TW18} by introducing an auxiliary very ample divisor and consider the K\"{a}hler-Einstein metric on $(S, \Delta+\frac{1-t}{m}H)$ for $t$ arbitrarily close to $1$. There seems to be some technical difficulty in producing such K\"{a}hler-Einstein metrics on a general singular K-semistable Fano pair.

Similar argument can be used to prove a result in the log-Calabi-Yau case:
\begin{thm}\label{thm-mainlogCY}
Let $(S, \Delta)$ be a log-Calabi-Yau pair that is log smooth. Let $L$ be any ample $\bQ$-line bundle on $S$. Then the orbifold tangent sheaf $T_S(-\log(\Delta))$ is slope semistable with respect to $L$.

Moreover, let $\scr{E}=\scr{E}(L)$ be the extension sheaf of $T_S(-\log(\Delta))$ by $\cO_S$ with the extension class $c_1(L)$. Then $\scr{E}$ is slope semistable with respect to $L$.
\end{thm}

Note that the first statement is well known if $\Delta=0$, while the second statement seems to be new even if $\Delta=0$.

We will give two applications of the above results. The first one is to study a question of Borbon-Spotti on the relation between the volume densities of Calabi-Yau metrics on log surfaces and the local Euler numbers of log canonical singularities (see \cite{Meg99,Lan03}). In \cite{Lan03}, Langer introduced local Euler numbers for general log canonical surface singularities and used it to prove a Miyaoka-Yau inequality for any log canonical surface. In an attempt to understand Langer's inequality using the K\"{a}hler-Einstein metric on a log canonical surface, 
Borbon-Spotti conjectured recently in \cite{BS17} that the volume densities of the singular K\"{a}hler-Einstein metrics should match Langer's local Euler numbers (at least for log terminal surface singularities). They verified this in special examples by comparing the known values of both sides. 
In another related direction, in a series of papers (see \cite{Li15, LL16, LX16} and the references therein), a new notion of normalized volumes of klt singularities has been developed. It has been proved that the normalized volume is equal to the volume density up to a factor $(\dim X)^{\dim X}$ for any point $(X,x)$ that lives on a Gromov-Hausdorff limit of smooth K\"{a}hler-Einstein manifolds (\cite{HS16, LX17}). In view of this connection, one can formulate a purely algebraic problem about two algebraic invariants of the singularities.
This problem was already posed by Borbon-Spotti at least in the log terminal case. We formulate the following form by including one of Langer's expectations (see \cite[p.381]{Lan03}):
\begin{conj}[{see \cite[p.37]{BS17}}]\label{conj}
Let $(X, D, x)$ be a germ of log canonical surface singularity. Then we have
\begin{equation}
e_{\rm orb}(x, X, D)=
\left\{
\begin{array}{cl}
\frac{1}{4}\hvol(x, X,D), & \text{ if } (X,D) \text{ is log terminal };\\
0, & \text{ if } (X, D) \text{ is not log terminal}.
\end{array}
\right. 
\end{equation}
\end{conj}
We refer to Definition \ref{defn-locEuler} and Definition \ref{defn-hvol} for the definition of the two sides. 

In this paper,  we will confirm this conjecture for log canonical cone singularities.
\begin{defn}(see \cite{LX17})\label{defn-logcone}
A {\it good $\bC^*$ action} on a log pair $(X, D)$ is a $\bC^*$-action on $X$ that preserves the divisor $D$ and has a unique attractive fixed point $x$ which is in the closure of any $\bC^*$-orbit on $X$. In this case, $(X^\circ, D^{\circ}):=(X\setminus\{x\}, D\setminus\{x\})$ is a $\bC^*$-Seifert bundle over the quotient orbifold $(X, D)/\bC^*:=(X^\circ, D^\circ)/\bC^*=(S, \Delta)$. Note that $\Delta$ consists of both the quotient of $D$ and the orbifold locus of the quotient map $X^\circ\rightarrow S$. We will also say that $(X,D)$ is a log orbifold cone (or simply a log cone) over $(S, \Delta)$. 

A log cone $(X, D)$ is a log-Fano cone if $(X,D)$ has klt singularities and $(S, \Delta)$ is a log-Fano pair.

A log cone $(X, D)$ is a log-CY cone if $(X,D)$ has log canonical singularities and $(S, \Delta)$ is log-CY.
\end{defn}

\begin{exmp}\label{exmp-C2D}
If $X=\bC^2$ and the good $\bC^*$-action associated to the weight $(a,b)$ ($a,b\in \bN, {\rm gcd}(a,b)=1$), then any log cone singularity covered by Definition \ref{defn-logcone} is of the form pair $(\bC^2, D, 0)$ where $D=c_0 \{z_2=0\}+c_\infty \{z_1=0\}+\sum_i c_i D_i$ where $D_i=\{u_i z_1^b- z_2^a=0\}$. The corresponding quotient is given by:
\begin{equation}
(S, \Delta)=\left(\bP^1, \left(\frac{c_0}{a}+\frac{a-1}{a}\right)\{0\}+\left(\frac{c_\infty}{b}+\frac{b-1}{b}\right) \{\infty\}+\sum_i c_i \{u_i\}\right),
\end{equation} 
and the orbifold line bundle is given by the $\bQ$-Weil divisor $L:=-\lambda^{-1} (K_S+\Delta)$ where $\lambda=a+b-c_0b-c_\infty a-ab \sum_i c_i$. Note that $\deg_{\bP^1}(L)=1/(ab)$. Of course, $(S, \Delta)$ can also be obtained by using weighted blow up. In other words, if $(Y, D_Y, E)\rightarrow (\bC^2, D, 0)$ is the weighted blow up with weight $(a,b)$, where $E\cong \bP^1$ is the exceptional divisor, then we have 
$
(S, \Delta)=(E, {\rm Diff}_{E}(D_Y))
$
(see \cite{Kol13, LX16}).
 
\end{exmp}

Before we state our next results, we explain very roughly why the stability of extension sheaves in Theorem \ref{thm-main}-Theorem \ref{thm-mainlogCY} can be applied to the log canonical cone case of conjecture \ref{conj}. 
In \cite{Lan03}, Langer defined local Euler numbers by using local second Chern classes of sheaves of logarithmic co-tangent sheaves on (coverings of) log resolutions. Based on a previous calculation of Wahl, he showed that such local second Chern classes can be effectively calculated when we have a cone singularity such that the resolution is given by the standard blow up of the vertex of the cone and the sheaf on the blow-up is the pull back of a sheaf on the base. In this case the local second Chern class is related to the semistability of the sheaf on the base (see Theorem \ref{thm-coneuler}). On the other hand, Wahl proved a basic fact in \cite[Proposition 3.3]{Wah76} that in the cone case, the logarithmic cotangent sheaf of the standard blow-up is exactly the pull back of the extension sheaf of the co-tangent sheaf of the base with the extension class given by the corresponding polarization. Our main observation is that these two ingredients can be combined and generalized to the logarithmic case. As a consequence, this allows us to apply Theorem \ref{thm-main}-\ref{thm-mainlogCY} and Langer-Wahl's formula to calculate the local Euler class when we have a K-semistable log Fano cone or a log Calabi-Yau cone. On the other hand, the normalized volumes of semistable log Fano cone singularities have been calculate in full generality in \cite{LL16, LX16} (see Theorem \ref{thm-semivol}-\ref{thm-degvol}). So we can compare and confirm \ref{conj} for these log cone singularities. Next we will describe the results.

Note that the quotient of a 2-dimensional log-Fano cone by its $\bC^*$-action is always a marked Riemann sphere
$(S, \Delta):=(\bP^1, \sum_i \delta_i p_i)$ satisfying (see Lemma \ref{lem-coneklt}): 
\begin{equation}
\delta_i\in (0,1]\cap \bQ \text{ and } \sum_i \delta_i< 2.
\end{equation} 
In this case, the K-semistabililty of $(S, \Delta)$ can be completely characterized by the (closed) Troyanov condition (see \cite[Example 2]{Li12}, \cite{Fuj17}):
\[
\left(\bP^1, \sum_i \delta_ip_i\right) \text{ is K-semistable } \quad \Longleftrightarrow \quad \sum_{j\neq i} \delta_j\ge \delta_i, \forall i.
\]
\begin{prop}\label{prop-conj}
Let $(X, D, x)$ be a log-terminal singularity with a good $\bC^*$-action such that it is an orbifold cone over $(\bP^1, \sum_i \delta_i p_i)$. If $(\bP^1, \sum_i \delta_i p_i)$ is K-semistable, then 
the Conjecture \ref{conj} holds true. 
\end{prop}
Combined with Langer's calculation in (\cite[Lemma 8.8]{Lan03}), we then see that Conjecture \ref{conj} is indeed true for all $2$-dimensional log-Fano cones without assuming that $(\bP^1, \sum_i\delta_i p_i)$ is K-semistable: 
\begin{cor}\label{cor-conetrue}
Let $(X, D, x)$ be a $2$-dimensional log-Fano cone singularity. Then the Conjecture \ref{conj} is true. 
\end{cor}
By similar argument we can apply Theorem \ref{thm-mainlogCY} to confirm Conjecture \ref{conj} for log-CY cone singularities:
\begin{prop}\label{prop-conjCY}
Assume $(X, D, x)$ is a $2$-dimensional log-CY cone. Then the Conjecture \ref{conj} holds true, i.e. $e_{\orb}(X, D, x)=0$.
\end{prop}
We expect the results in \ref{prop-conj}-\ref{prop-conjCY} to be useful to attacking the general case combined with some degeneration/deformation techniques.
To highlight our results, note that Proposition \ref{prop-conj}-\ref{prop-conjCY} in particular answers a question in \cite[Remark on p. 387]{Lan03} and completes the computation of local Euler numbers of line arrangements on $\bC^2$. In other words, we now know that the inequality for the last case considered in \cite[Theorem 8.3]{Lan03} is indeed an identity:
\begin{cor}
Let $L_1, \dots, L_n$ be $m$ distinct lines in $\bC^2$ passing through $0$. Let $D=\sum_{i=1}^m \delta_i L_i$, where $0\le \delta_1\le \delta_2\le \dots \le \delta_m\le 1$, and $\delta=\sum_{i=1}^m \delta_i$. If 
$2\delta_m\le \delta\le 2$, then $e_{\orb}(0, \bC^2, D)=\frac{(2-\delta)^2}{4}$.
\end{cor}
Other immediate consequences of Theorem \ref{thm-main} and Theorem \ref{thm-mainlogCY} are the following Chern number inequalities for K-semistable log-smooth
log-Fano pairs, and Calabi-Yau pairs. These generalize Chern number inequalities of Song-Wang \cite{SW12} and should be thought of as the log-Fano/log-Calabi-Yau version of the Miyaoka-Yau inequality. Indeed, the use of Higgs bundle in a proof of Miyaoka-Yau's inequality (see \cite[pp.149]{Tia94}, \cite{GKPT15, GT16} and the reference therein) for the log general type case is mirrored here by the use of the extension sheaf from Theorem \ref{thm-main} and Theorem \ref{thm-mainlogCY}. 
To state the result, we recall that according to \cite[Lemma 2.4]{Tia94} and \cite[Example 3.6]{GT16}), $c_i(S, \Delta)$, $i=1,2$ for log smooth pair $(S, \Delta)$ 
are given by the following expressions:
\begin{eqnarray}\label{eq-logChern}
c_1(S, \Delta)&=&c_1(S)+\Delta\; ,\nonumber\\
c_2(S, \Delta)&=&c_2(S)+K_S\cdot \Delta+\sum_i \delta_i \Delta_i^2+\sum_{i<j}\delta_i\delta_j \Delta_i\cdot \Delta_j.
\end{eqnarray}
We then have the following results:
\begin{thm}\label{thm-Chineq}
Let $(S, \Delta)$ be a log-smooth log-Fano pair. Assume $(S, \Delta)$ is K-semistable. Then we have the following Chern-number inequality:
\begin{equation}
\left(2(n+1)c_2(S, \Delta)-n \cdot c_1^2(S, \Delta)\right)\cdot (-K_S-\Delta)^{n-2}\ge 0,
\end{equation}
where $c_i(S, \Delta), i=1,2$ are logarithmic Chern classes appearing in \eqref{eq-logChern}. 
\end{thm}
\begin{thm}\label{thm-ChineqCY}
Let $(S, \Delta)$ be a log-smooth log-Calabi-Yau pair. Let $L$ be any nef line bundle over $S$. Then we have the following Chern-number inequality:
\begin{equation}
c_2(S, \Delta)\cdot L^{n-2}\ge 0.
\end{equation}
\end{thm}
We note that the Calabi-Yau case Theorem \ref{thm-ChineqCY} also follows from the work in \cite{GT16}. As in \cite[Theorem B]{GT16}, the log-smooth condition could be weakened under suitable assumptions of the pair (see also remark \ref{rem-ChineqCY}).

We remark that although the statements of the above theorems are purely algebraic, their proofs depend heavily on the use of K\"{a}hler-Einstein metrics on log-Fano or log-Calabi-Yau pairs. It would be interesting to give purely algebraic proofs of the above results.

Now we sketch the organization of this paper. In the next section, we recall a construction of the pull back of orbifold (co-)tangent sheaves after taking log resolutions and ramified coverings.
This is well known and our exposition is inspired by \cite{Lan03, GT16}. With these sheaves, we state Langer's definition of local Euler numbers for log canonical surface singularities. In section \ref{sec-cone} we specialize to the case of log canonical cone singularities. Here we generalize a result of Wahl identifying the sheaf of logarithmic 1-forms on the standard blow-up of cone singularity with the pull back of an extension sheaf on the base. This is a bridge from Theorem \ref{thm-main}-\ref{thm-mainlogCY} to Theorem \ref{prop-conj}-\ref{prop-conjCY} because the objects studied in Theorem \ref{thm-main}-\ref{thm-mainlogCY} are just examples of such extension sheaf. In section \ref{sec-logTian}, we then extend Tian's semistability result and prove Theorem \ref{thm-main} and Theorem \ref{thm-mainlogCY}. Here we use similar argument as \cite{GT16} to deal with the technical difficulty caused by the conical singularities of K\"{a}hler-Einstein metrics on log pairs. In the log-Fano cone case, we will first prove the polystable case in Theorem \ref{thm-poly} and then use a perturbative approach to deal with the K-semistable case in Theorem \ref{thm-semi}. In section \ref{sec-singularstable}, we also prove a generalization of Theorem \ref{thm-main} for a class of singular log-Fano pairs. In section \ref{sec-logCY}, we prove Theorem \ref{thm-mainlogCY}.

In section \ref{sec-volEuler}, we recall the normalized volume of log terminal singularities. Combining the results from previous sections and the properties/calculations of the invariants for log canonical cone singularities, we complete the proof of Proposition \ref{prop-conj}, Corollary \ref{cor-conetrue} and 
Proposition \ref{prop-conjCY}.

In section \ref{sec-Chineq}, we prove Theorem \ref{thm-Chineq} (resp. \ref{thm-ChineqCY}) by combining Theorem \ref{thm-main} (resp. \ref{thm-mainlogCY}) and the Bogomolov-Gieseker inequality for slope semistable vector bundles.

\bigskip
\noindent
{\bf Acknowledgement:} 
The author is partially supported by NSF (Grant No. DMS-1405936) and an Alfred P. Sloan research fellowship. The author would like to thank Martin de Borbon and Christiano Spotti for 
useful comments and the suggestion of adding the example \ref{exmp-C2D}, and to thank Henri Guenancia and Behrouz Taji for their interest and especially to Behrouz Taji for very helpful comments and suggestions about orbifold structures. His thanks also go to Yuchen Liu, Xiaowei Wang and Chenyang Xu for helpful discussions. The author would like to thank Professor Gang Tian for his interest in this work and constant support through the years.

\section{Pull back of orbifold (co-)tangent sheaves}

\subsection{General constructions}\label{sec-genlog}
We first define the pull back of the sheaf of logarithmic 1-forms along $\bQ$-divisors by combining the constructions in \cite{Lan03, GT16}.  Let $(X, x)$ be an $(n+1)$-dimensional germ of normal affine variety and let $D=\sum_{i} \delta_i D_i$ be a $\bQ$-divisor with $\delta_i\in [0,1]$. 
Choose a log resolution $\mu_X: (\tilde{X}, \tilde{D}=(\mu_X)_*^{-1}D, E_x)\rightarrow (X, D, x)$ where $E_x=\sum_j E_j$ is a simple normal crossing divisor that is contracted to $x$. 

By Kawamata's covering lemma, we can choose a very ample divisor $H_{\til{X}}$ over $\tilde{X}$ such that $H_{\til{X}}+\til{D}+E_x$ has simple normal crossings and construct a finite morphism $\sigma_{\tilde{X}}: \tilde{Y}\rightarrow \tilde{X}$ of degree $N$ which is a ramified Galois cover with group $G$ and it satisfies:
\begin{enumerate}
\item $\sigma_{\til{X}}$ is \'{e}tale over the complement of $\sum_{i} \tilde{D}_i+H_{\til{X}}$.
\item $\sigma_{\tilde{X}}^*(\til{D}+H_{\til{X}}+E_x)$ is a simple normal crossing Weil divisor.
\item Near any point $y_0\in \tilde{Y}$, there exists a $G$-invariant open set $U\ni y_0$, a system of coordiantes $\{w_k\}$ centered at $y_0$, a system of coordinates $\{z_k\}$ near $\sigma_{\til{X}}(y_0)$ and an integer $p=p(y_0)$ such that, with respect to these coordinates, the map $\sigma_{\til{X}}$ is locally expressed as:
\begin{eqnarray*}
&&(w_1,\dots, w_p, w_{p+1}, \dots, w_{n+1}) \mapsto \\
&& \hskip 2cm (w_1, \cdots, w_p, w_{p+1}^N,  \dots, w^N_{n+1})=(z_1,\cdots, z_{n+1}).
\end{eqnarray*}
\end{enumerate} 
Let $\sigma_X: Y\rightarrow X$ be the Stein factorization of the composition $\mu_X\circ \sigma_{\til{X}}: \tilde{Y}\rightarrow X$. Then $\sigma_X^*D$ is an integral Weil divisor and we have the commutative diagram:
 \[
 \begin{tikzcd}
  \tilde{Y}\arrow{r}{\mu_Y}\arrow{d}[swap]{\sigma_{\til{X}}} & Y
  \arrow{d}{\sigma_X}\\
  \til{X} \arrow{r}{\mu_X} & X
 \end{tikzcd}
 \]
Denote $D'_i=\sigma_{\til{X}}^{-1}D_i$ and $H'_{\til{X}}=\sigma_{\til{X}}^{-1}H_{\til{X}}$. 
By construction, the ramification divisor of $\sigma_{\til{X}}$ is equal to $(N-1)\sum_i D'_i+(N-1)H'_{\til{X}}$. In other words we have the following identity:
\[
K_{\tilde{Y}}=\sigma_{\til{X}}^*K_{\tilde{X}}+\sum_i  (N-1)D'_i+(N-1)H'_{\til{X}}
\]
Hence the pull back of the log canonical divisor $K_{\tilde{X}}+\til{D}+E_x$ under $\sigma_{\til{X}}$ is given by:
\begin{eqnarray}\label{eq-HurX}
\sigma_{\til{X}}^*\left(K_{\til{X}}+\tilde{D}+E_x\right)
&=&K_{\til{Y}}-(N-1)\left(\sum_i D'_i+H'_{\til{X}}\right)+\sum_i N\delta_i D'_i+\sum_j E'_j\nonumber\\
&=&K_{\til{Y}}+\sum_i(1-N+N\delta_i)D'_i+\sum_j E'_j+(1-N)H'_{\til{X}}\\
&=:&K_{\til{Y}}+G.
\end{eqnarray}
\begin{nota}
Write $G=\sum_i d_i G_i$ where each irreducible component of ${\rm supp}(G)$ is either equal to $D'_i$, $E'_j$ or $H'_{\til{X}}$, and $d_i$ is equal to $1-N+N\delta_i$, $1$ or $1-N$ correspondingly.

\noindent
Note that if $\delta_i=1$, then $d_i=1-N+N\delta_i=1$, while if $\delta_i\in [0,1)$, then $d_i=1-N+N\delta_i\le 0$.
\end{nota}
In the language of \cite{GT16} (see Definition \ref{defn-adapted}), $\sigma_{\til{X}}$ is a global {\it adapted morphism} defining an orbifold structure on the pair $(\til{X}, \til{D}+E_x)$. This explains the terminology in the following definition.
\begin{defn}\label{defn-virXD}
With the above notations, 
the pull back of the orbifold tangent sheaf of $(\til{X}, \til{D}+E_x)$ with respect to $\sigma_{\til{X}}$, denoted by
$\sigma_{\til{X}}^*\Omega_{\til{X}}^1(\log(\tilde{D}+E_x))$, is defined to be the $\cO_{\til{Y}}$-module locally given by:
\[
\sum_{i=p+1}^{n+1} \cO_{\til{Y}} w_{i}^{-d_i}d w_i+\sigma_{\til{X}}^*\Omega^1_{\til{X}}.
\]
We will also denote this sheaf by $\Omega^1_{\til{Y}}(\log(G))$ since most of the time we will calculate directly over $\til{Y}$ . 

Dually, the pull back of the orbifold tangent sheaf of $(\til{X}, \til{D}+E_x)$ with respect to $\sigma_{\til{X}}$, denoted by $\sigma_{\til{X}}^*T_{\til{X}}(-\log(\til{D}+E_x))$, is defined to be the $\cO_{\til{Y}}$-module locally given by:
\[
\sum_{i=p+1}^{n+1} \cO_{\til{Y}} w_{i}^{d_i}\frac{\partial}{\partial w_i}+\sum_{i=1}^p \cO_{\til{Y}}\frac{\partial}{\partial w_i}.
\]
We will also denote this sheaf by $T_{\til{Y}}(-\log(G))$.
\end{defn}
\begin{rem}
\begin{enumerate}
\item
We have following identities which shows that the above definition is the same as in \cite[2]{Lan03}:
\[
w_i^{-d_i}dw_i=\frac{w_i^{N\delta_i-d_i}dw_i}{w_i^{N\delta_i}}=\frac{w_i^{N-1}dw_i}{w_i^{N\delta_i}}=\frac{\sigma_{\til{X}}^*dz_i}{w_i^{N\delta_i}}.
\]
\item
By definition, the above sheaves depend on the choice of the log resolution $(\til{X}, \til{D})\rightarrow (X, D)$ and the ramified covering $\til{Y}\rightarrow \til{X}$.
However they transform naturally if different choices are made in the construction. So these sheaves should be considered as representations of orbifold (co-)tangent sheaves associated to the original pair $(X, D)$ (called ``virtual sheaves" in \cite[2]{Lan03}). See also appendix \ref{app-orbifold}.
\end{enumerate}
\end{rem}

\subsection{Euler numbers for log canonical surface singularities}
In this section, we assume that $(X,D, x)$ is a log canonical surface singularity and carry out the construction described in the previous section. 
\begin{dfpr}[\cite{Lan03}]\label{defn-locEuler}
Let $(X, D, x)$ be a log canonical surface singularity. With the notations in previous section,
the local Euler number of the pair $(X, D)$ at $x$ is defined by (see the following remark):
\begin{equation}\label{eq-locEuler}
e_{\rm orb}(x, X,D)=-\frac{c_2\left(\mu_Y, \sigma_{\til{X}}^*\Omega^1_{\tilde{X}}((\mu_{X})_*^{-1}D+E_x)\right)}{\deg \sigma}.
\end{equation}
This is well defined and does not depend on the choice of the log resolution $\mu_X: (\til{X}, \mu_*^{-1}(D), E_x)\rightarrow (X, D, x)$ or the covering $\sigma_{\til{X}}: \til{Y}\rightarrow \til{X}$.
\end{dfpr}
The numerical invariant $c_2(\mu_Y, \scr{F}):=c_2(\mu_Y, \scr{F}^{**})$ on the right-hand-side of \eqref{eq-locEuler} (called local Chern class) was defined in \cite{Lan03} for any locally free sheaf $\scr{F}$ over $\tilde{Y}$.
In the current paper, it's not very important what is the exact formula for the $c_2$. We just mention that this term arises in Langer's proof of Miyaoka-Yau's inequality for general log canonical surfaces and it's conjectured to coincide with Wahl's local second Chern class from \cite{Wah93} when $D=0$.  This is indeed the case for surface cone singularities and follows essentially from Wahl's calculations in \cite{Wah93}. Here we only need the following formula from \cite{Lan03}, which motivates us to consider the case of cone singularities in the following subsection.

\begin{thm}[{\cite[Theorem 1.10]{Lan03}}]\label{thm-coneuler}
Let $\scr{E}$ be a rank-2 vector bundle on a smooth projective curve $C$ and let $L$ be a line bundle with degree $d>0$. Set $e=\det \scr{E}$ and
\begin{equation}
\bar{s}=\bar{s}(\scr{E})=\max\left(\frac{1}{2}e, \max\{\deg\scr{L}: \scr{L} \subset \scr{E}\} \right).
\end{equation}
Let $\tilde{X}$ be the total space of a line bundle $L^{-1}$ and let $\pi: \tilde{X}\rightarrow C$ the canonical projection. Let $\mu_X: \tilde{X}\rightarrow X$ be the contraction of the zero section of $L^{-1}$. Then
\[
c_2(\mu_X, \pi^*\scr{E})=-\bar{s}(e-\bar{s})/d \ge -\frac{e^2}{4d}.
\]
In particular, if $\scr{E}$ is semistable then $c_2(\mu, \pi^*\scr{E})=-\frac{e^2}{4d}=-\frac{c_1(\scr{E})^2}{4d}$.
\end{thm}
A. Langer in \cite[section 8]{Lan03} used the above formula to calculate $e_\orb$ for line arrangements $(\bC^2, \sum_{i=1}^m \delta_i L_i)$ with $m\le 3$, which was used in turn to calculate the $e_{\orb}$ for any log canonical pair $(X,D)$ with a fractional boundary (\cite[section 9]{Lan03}). As mentioned in the introduction, our semistability result will allow to calculate Langer's local Euler numbers for line arrangements consisting of any number of lines. 

We shall need one important property of local Euler numbers:
\begin{lem}[{\cite[Lemma 7.1]{Lan03}}]\label{lem-Eulerdeg}
If $\sigma: (X, D, x)\rightarrow (Z, D_Z, z)$ is a finite proper morphism of normal proper surface germs and $K_{X}+D=\sigma^*(K_Z+D_Z)$ for some boundary $\bQ$-divisor $D_Z$ on $Z$, then
\[
e_{\orb}(x, X,D)=\deg(\sigma) \cdot e_{\orb}(z, Z,D_Z).
\]
\end{lem}

\subsection{Log cone singularity}\label{sec-cone}
Here we specialize the constructions in \ref{sec-genlog} to the case of cone singularities. Let $S$ be a normal projective variety of dimension $n$, $L$ an ample Cartier divisor on $S$ and $X=C(S, L)={\rm Spec}_{\bC}\left(\oplus_{k=0}^{+\infty} H^0(S, kL)\right)$ the corresponding affine cone. Let $\Delta=\sum_i \delta_i \Delta_i$ be an effective $\bQ$-divisor and $D=C(\Delta,L)$ the corresponding $\bQ$-divisor on $C(S, L)$. We will assume $-(K_S+\Delta)$ is $\bQ$-Cartier. Let $x\in X$ denote the closed point of the cone defined by the maximal ideal 
$\bigoplus_{k=1}^{+\infty}H^0(S, kL)$. Then a basic fact for us is:
\begin{lem}[{\cite[Lemma 3.1]{Kol13}}]\label{lem-coneklt}
With the above notations, $(X, D)$ has klt singularities if and only if $-(K_S+\Delta)=\lambda L$ with $\lambda>0$ and $(S, \Delta)$ is klt. $(X, D)$ is has log canonical singularities if and only if $-(K_S+\Delta)=\lambda L$ with $\lambda \ge 0$ and $(S, \Delta)$ is log canonical.
\end{lem}
Let $\hat{X}\rightarrow X$ denote the total space of the line bundle $L^{-1}$ and $\pi_S: \hat{X}\rightarrow S$ denote the natural projection. Let $\hat{\mu}_X: \hat{X}\rightarrow X$ denote the birational contraction of the zero section of $L^{-1}$. The unique exceptional divisor of $\hat{\mu}_X$ is isomorphic to $S$. Now we choose a log resolution $\mu_S: (\tilde{S}, \til{\Delta})\rightarrow (S, \Delta)$ 
and $(\tilde{X}, \tilde{D}):=(\hat{X}, \pi_S^{-1}(\Delta))\times_S\til{S}$ with a natural projection $\mu_{\hat{X}}: \til{X}\rightarrow \hat{X}$. $\til{X}$ is just the total space of the line bundle $\mu_S^*L$ with the natural projection $\pi_{\til{S}}: \til{X}\rightarrow \til{S}$. The natural morphism $\mu_X:=\hat{\mu}_X\circ \mu_{\hat{X}}: (\til{X}, \til{D}) \rightarrow (X, D)$ is a log resolution whose exceptional divisor over $x$ is given by:
\[
E_x=\mu_X^{-1}(x)\cong \tilde{S}. 
\]
Now we apply Kawamata's covering lemma to $(\tilde{S}, \tilde{\Delta})$ as in the previous subsection. In other words, we choose a very ample divisor $H$ such that the support of $\til{\Delta}+H$ has simple normal crossings and construct a finite morphism $\sigma_{\til{S}}: S'\rightarrow \til{S}$ of degree $N$ which is a ramified Galois cover with group $G$ and it satisfies:
\begin{enumerate}
\item[(i)] $\sigma_{\til{S}}$ is \'{e}tale over the complement of $\sum_{i}\til{\Delta}_i+H$.
\item[(ii)] $\sigma_{\til{S}}^*(\til{\Delta}+H)$ is a simple normal crossing Weil divisor.
\item[(iii)] Near any point $y_0\in S'$, there exists a $G$-invariant open set $U\ni y_0$, a system of coordiantes $\{w_k\}$ centered at $y_0$, a system of coordinates $\{z_k\}$ near $\sigma_{\til{S}}(y_0)$ and an integer $p=p(y_0)$ such that, with respect to these coordinates, the map $\sigma_{\til{S}}$ is locally expressed as:
\[
(w_1, \dots, w_n) \mapsto (w_1,\dots, w_p, w_{p+1}^N, \dots,  w^N_{n})=(z_1,\dots, z_p, z_{p+1}, \dots, z_{n}).
\]
\end{enumerate}
We denote fiber product $S'\times_{\til{S}}\til{X}$ by $\til{Y}$. Then $\til{Y}$ is nothing but the total space of $\sigma_{\til{S}}^*\mu_{S}^*L$, and we have the following commutative diagram:
\begin{equation}\label{eq-CDcone}
\begin{CD}
\tilde{Y} @>\sigma_{\til{X}}>>  \tilde{X} @>\mu_{\hat{X}}>> \hat{X} @>\hat{\mu}_X >> X\\
@V\pi_{S'}VV  @V\pi_{\til{S}} VV @V\pi_S VV \\
S' @>\sigma_{\til{S}}>> \til{S} @>\mu_S>> S  
\end{CD}
\end{equation}
 
As before, we have the identity: \begin{eqnarray}\label{eq-HurX}
\sigma_{\til{S}}^*\left(K_{\til{S}}+\tilde{\Delta}\right)
&=&K_{S'}+\sum_i(1-N+N\delta_i)\Delta'_i-(N-1)H'\\
&=:&K_{S'}+B.
\end{eqnarray}
\begin{nota}\label{nota-B}
Write $B=\sum_i d_i B_i$. Each irreducible component of ${\rm supp}(B)$ is equal to $\Delta'_i$ or $H'$, and $d_i$ is equal to $1-N+N\delta_i$ or $1-N$ correspondingly. 
\end{nota}
Similar to \ref{defn-virXD}, we define:
\begin{defn}\label{defn-virSDel}
The pull back of the orbifold cotangent sheaf of $(\til{S}, \til{\Delta})$ with respect to $\sigma_{\til{S}}$, denoted by $\sigma_{\til{S}}^*\Omega_{\til{S}}^1(\log(\tilde{\Delta}))$, is defined to be the $\cO_{S'}$-module locally given by:
\[
\sum_{i=p+1}^n  \cO_{S'}\cdot w_i^{-d_i} dw_i+\sigma_{\til{S}}^*\Omega^1_{\til{S}}.
\]
We will also denote such a sheaf by $\Omega^1_{S'}(\log(B))$.

Dually, the pull back of the orbifold tangent sheaf of $(\til{S}, \til{\Delta})$ with respect to $\sigma_{\til{S}}$, denoted by $\sigma_{\til{S}}^*T_{\til{S}}(-\log(\til{\Delta}))$, is defined to be the $\cO_{S'}$-module locally given by:
\[
\sum_{i=p+1}^n \cO_{S'}\cdot w_i^{d_i}\frac{\partial}{\partial w_i}+\sum_{i=1}^p \cO_{S'}\cdot \frac{\partial }{\partial w_i}.
\]
We will also denote such a sheaf by $T_{S'}(-\log(B))$.
\end{defn}

By definition, there is a natural injection of the sheaves $\sigma_{\til{S}}^*\Omega_{\til{S}}^1\hookrightarrow \sigma_{\til{S}}^*\Omega_{\til{S}}^1(\log(\til{\Delta}))$. We will denote by $\phi$ the
induced map on the cohomologies: 
\begin{equation}\label{eq-injcoh}
\phi: H^1(S', \sigma_{\til{S}}^*\Omega_{\til{S}}^1)\rightarrow H^1(S', \sigma_{\til{S}}^*\Omega^1_{\til{S}}(\log(\til{\Delta})))=H^1(S', \Omega^1_{S'}(\log(B))).
\end{equation} 

On the other hand, in Definition \ref{defn-virXD} of the previous subsection, we have defined \\
$\sigma_{\tilde{X}}^*\Omega^1_{\til{X}}(\log(\til{D}+E_x))=\Omega^1_{X'}(\log(G))$
and its dual $\sigma_{\til{X}}^*T_{\til{X}}(-\log(\til{D}+E_x))=T_{X'}(-\log(G))$.
The main goal in this section is to prove the following result which generalizes \cite[Proposition 3.3]{Wah76}.
\begin{prop}\label{prop-base2cone}
With the above notations, there is an exact sequence on $\til{Y}$:
\begin{equation}
\xymatrix 
{
0 \ar[r]  & \pi_{S'}^*\sigma_{\til{S}}^*\Omega^1_{\til{S}}(\log(\til{\Delta})) 
\ar@{=}[d] \ar[r]  &  \sigma_{\til{X}}^* \Omega^1_{\til{X}}(\log(\til{D}+E_x)) \ar[r] \ar@{=}[d] & 
\cO_{\til{Y}}
 \ar[r] \ar@{=}[d] & 0 \\
0 \ar[r] &  \pi_{S'}^*\Omega^1_{S'}(\log(B)) \ar[r]  & \Omega_{\til{Y}}^1(\log(G))  \ar[r]
     & \cO_{\til{Y}}  \ar[r] & 0
}
\end{equation}

If we let $\scr{E}_{S'}^{\vee}=\sigma_{\til{X}}^*\Omega^1_{\tilde{X}}(\log(\til{D}+E_x))\otimes_{\cO_{\til{Y}}}{\cO_{S'}}$, then the above sequence is the pull back via $\pi_{S'}$ via the following exact sequence on $S'$:
\begin{equation}\label{eq-ExtS}
\xymatrix
{
0 \ar[r] & \sigma_{\til{S}}^*\Omega^1_{\til{S}}(\log(\til{\Delta})) \ar@{=}[d] \ar[r] & \scr{E}_{S'}^{\vee} \ar@{=}[d] \ar[r] & \cO_{S'} \ar@{=}[d] \ar[r] & 0\\
0 \ar[r] & \Omega^1_{S'}(\log(B)) \ar[r] & \scr{E}_{S'}^{\vee} \ar[r] &  \cO_{S'} \ar[r] & 0.
}
\end{equation}
Moreover, the extension class of the exact sequence \eqref{eq-ExtS} is given by $\Phi(c_1(L))$ where $\Phi$ is the composition of the following natural homomorphism of cohomology groups (
$\delta$ is the natural connecting morphism for the exact sequence and
$\phi$ was given in \eqref{eq-injcoh}):
\begin{eqnarray*}
&&H^1(\til{S}, \cO_{\til{S}}^*)\stackrel{{\bf \delta}}{\longrightarrow} H^1(\til{S}, \Omega_{\til{S}}^1)\stackrel{\sigma_{\til{S}}^*}{\longrightarrow} H^1(S', \sigma_{\til{S}}^*\Omega_{\til{S}}^1)\\
&&\hskip 5cm \stackrel{\phi}{\longrightarrow} H^1(S', \sigma_{\til{S}}^*\Omega^1_{\til{S}}(\log(\til{\Delta})))=H^1(S', \Omega^1_{S'}(\log(B))).
\end{eqnarray*}
\end{prop}
\begin{rem}
If $\til{\Delta}=0$, then we get back the result in \cite[Proposition 3.3]{Wah76}, whose proof will be generalized in the following proof.
\end{rem}
\begin{proof}
We choose an affine variable $\xi$ along the fibre of the line bundle: $\sigma_{\til{S}}^*\mu_S^*L$. Then by definition, upstairs on $\tilde{Y}$, $\sigma_{\til{X}}^*T_{\tilde{X}}(-\log(\til{D}+E_x))$ is locally spanned by:
\[
\frac{\partial}{\partial w_1},\dots, \frac{\partial}{\partial w_p},\quad w_{p+1}^{d_{p+1}}\frac{\partial}{\partial w_{p+1}} \dots, w_n^{d_n}\frac{\partial}{\partial w_n}, \xi\frac{\partial}{\partial \xi}.
\]
Dually $\sigma_{\til{X}}^*\Omega^1_{\tilde{X}}(\log(\til{D}+E_x))$ is spanned by:
\[
d w_1, \dots, d w_p,\quad w_{p+1}^{-d_{p+1}} dw_{p+1}, \dots,  w_n^{-d_n} \frac{d w_n}{w_n}, \frac{d\xi}{\xi}.
\]
Formally we have for $i=p+1,\dots, n$,
\begin{eqnarray*}
(w_i)^{d_i}\frac{\partial}{\partial w_i}&=& (z_i)^{d_i/N} N (w_i)^{N-1}\frac{\partial}{\partial z_i}=N (z_i)^{(d_i/N)+(1-N^{-1})}\frac{\partial}{\partial z_i}=N z_i^{\delta_i}\frac{\partial}{\partial z_i}.
\end{eqnarray*}
For the simplicity of notations, we let $\delta_1=\dots=\delta_p=0$ and write the generators above simply as:
\begin{eqnarray*}
&&\text{for } T_{\til{Y}}(-\log(G))=\sigma_{\til{X}}^*T_{\tilde{X}}(-\log(\til{D}+E_x)):\quad \left\{z_i^{\delta_i}\frac{\partial}{\partial z_i}, \xi \frac{\partial}{\partial \xi}\right\}=:\left\{\hat{D}_i, \hat{D}_\xi\right\}; \\
&& \text{for } \Omega^1_{\til{Y}}(\log G)=\sigma_{\til{X}}^*\Omega^1_{\tilde{X}}(\log(\til{D}+E_x)):\quad \left\{ \frac{dz_i}{z_i^{\delta_i}}, \frac{d\xi}{\xi}\right\}=:\left\{\hat{d}_i, \hat{d}_\xi\right\}.
\end{eqnarray*} 
Now consider coordinate change over $\til{S}$ on two overlapping coordinate neighborhoods $U_\alpha$ and $U_\beta$:
\[
\xi^\beta=f_{\alpha\beta}\xi^\alpha, \quad z_i^\beta=F_{i\alpha\beta}(z^\alpha).
\]
Then we can calculate the change of basis of $T_{\tilde{Y}}(-\log(G))$ over $Y$ (although we calculate formally on $X$, but they can all be pulled back to $Y$):
\begin{eqnarray*}
\hat{D}^\alpha_i&:=&(z_i^\alpha)^{\delta_i}\frac{\partial}{\partial z_i^\alpha}=(z_i^\alpha)^{\delta_i}\frac{\partial z_j^\beta}{\partial z_i^\alpha}\frac{\partial}{\partial z_j^\beta}+(z_i^\alpha)^{\delta_i} \frac{\partial f_{\alpha\beta}}{\partial z_i^\alpha}\xi^\alpha \frac{\partial}{\partial \xi^\beta}\\
&=& \frac{(z^\alpha_i)^{\delta_i}}{(z^\beta_j)^{\delta_j}}\frac{\partial z^\beta_j}{\partial z^\alpha_i}\hat{D}^\beta_j+(z^\alpha_i)^{\delta_i}f_{\alpha\beta}^{-1}\frac{\partial f_{\alpha\beta}}{\partial z^\alpha_i}\hat{D}^\beta_\xi; \\
\hat{D}^\beta_\xi&=&\xi^\beta\frac{\partial}{\partial \xi^\beta}=\xi^\alpha \frac{\partial}{\xi^\alpha}=:\hat{D}^\alpha_\xi.
\end{eqnarray*}
Dually we have the following change of basis for $\Omega^1_{\tilde{Y}}(\log(G))$:
\begin{eqnarray*}
&& \hat{d}^\beta_j=\frac{\partial z^\beta_j}{\partial z^\alpha_i} \frac{(z^\alpha_i)^{\delta_i}}{(z^\beta_j)^{\delta_j}}\hat{d}^\alpha_i, \quad
\frac{d\xi^\beta}{\xi^\beta}=\frac{d\xi^\alpha}{\xi^\alpha}+(z_i^\alpha)^{\delta_i} f_{\alpha\beta}^{-1}\frac{\partial f_{\alpha\beta}}{\partial z_i^\alpha}\hat{d}^\alpha_i.
\end{eqnarray*}
From the above change of basis, we easily get the following two exact sequences which are dual to each other:
\begin{eqnarray*}
&& 0 \longrightarrow \cO_{\til{Y}} \stackrel{p}{\longrightarrow} T_{\til{Y}}(-\log(G))\stackrel{q}{\longrightarrow} \pi_{S'}^*T_{S'}(-\log(B))\longrightarrow 0\\
&& \nonumber \\
&& 0\longrightarrow \pi_{S'}^*\Omega^1_{S'}(\log(B))\stackrel{q^\vee}{\longrightarrow} \Omega^1_{\til{Y}}(\log(G)) \stackrel{p^\vee}{\longrightarrow}\cO_{\til{Y}}\longrightarrow 0. 
\end{eqnarray*}
Indeed, the sheaf morphisms in the above exact sequences are locally given by:
\begin{eqnarray*}
&&p(1)=\hat{D}_\xi, \quad q(\hat{D}_i)=\hat{D}_i, \quad q(\hat{D}_\xi)=0; \\
&&q^\vee(\pi^*\hat{d}_i)=\hat{d}_i, \quad p^\vee(\hat{d}_i)=0, \quad p^\vee(\hat{d}_\xi)=1. 
\end{eqnarray*}

If we let $\scr{E}_{S'}=T_{\til{Y}}(-\log(G))\otimes_{\cO_{\til{Y}}} \cO_{S'}$, then these exact sequences are the pull-back via $\pi_{S'}$ of two dual exact sequences:
\begin{eqnarray}
&& 0\longrightarrow \cO_{S'}\longrightarrow \scr{E}_{S'} \longrightarrow T_{S'}(-\log(B))\longrightarrow 0\\
&& \nonumber \\
&& 0\longrightarrow \Omega^1_{S'}(\log(B))\longrightarrow \scr{E}_{S'}^\vee \longrightarrow \cO_{S'}\longrightarrow 0. \label{eq-cotext}
\end{eqnarray}

Moreover, the extension class of \eqref{eq-cotext} is given by the \v{C}ech cocycle:
\begin{eqnarray*}
c_{\alpha\beta}&=&\frac{d\xi^\beta}{\xi^\beta}-\frac{d\xi^\alpha}{\xi^\alpha}=(z^\alpha_i)^{\delta_i} f_{\alpha\beta}^{-1} \frac{\partial f_{\alpha\beta}}{\partial z^\alpha_i}\hat{d}^\alpha_i=\sigma_{\til{S}}^*\left(f_{\alpha\beta}^{-1}\frac{\partial f_{\alpha\beta}}{\partial z^\alpha_i}dz^\alpha_i\right). 
\end{eqnarray*}
Because $\{f_{\alpha\beta}^{-1}\frac{\partial f_{\alpha\beta}}{\partial z^\alpha_i}dz^\alpha_i\}$ is the image of $c_1(L)$ under the natural map
$\delta: H^1(\cO_{\til{S}}^*)\rightarrow H^1(\Omega^1_{\til{S}})$ and we have a natural map $\phi(\delta(c_1(L)))$ given in \eqref{eq-injcoh}, we easily get the last statement.
\end{proof}

From now on in this section, we assume that $(S, \Delta=\sum_i \delta_i\Delta_i)$ is a log smooth pair. Then in the construction of the previous section, we can choose $(\til{S}, \til{\Delta})=(S, \Delta)$ and $E_x=\mu_X^{-1}(x)\cong S$ and the commutative diagram in \eqref{eq-CDcone} simplifies to become:
\begin{equation}\label{eq-CDlogsmooth}
\begin{CD}
\tilde{Y} @>\sigma_{\til{X}}>>  \tilde{X} @>\mu_{X}>> X \\
@V\pi_{S'}VV  @V\pi_{S} VV  \\
S' @>\sigma_{S}>> S 
\end{CD}
\end{equation}
The following is then a corollary of Proposition \ref{prop-base2cone} in the case $(S, \Delta)$ is log smooth. Note that when $(S, \Delta)=(\bP^1, \sum_i \delta_ip_i)$, it also recovers the first statement of \cite[Lemma 8.8]{Lan03}.
\begin{cor}\label{cor-logsmooth}
Assume $(S, \Delta)$ is log smooth. With the same notations as Proposition \ref{prop-base2cone}, there is an exact sequence on $\til{Y}$:
\begin{equation}
\xymatrix{
0 \ar[r]  & \pi_{S'}^*\sigma_{S}^*\Omega^1_{S}(\log(\Delta))
\ar@{=}[d] \ar[r]  &   \sigma_{\til{X}}^* \Omega^1_{\til{X}}(\log(\til{D}+S)) \ar[r] \ar@{=}[d] & 
\cO_{\til{Y}}
 \ar[r] \ar@{=}[d] & 0 \\
0 \ar[r] &  \pi_{S'}^*\Omega^1_{S'}(\log(B)) \ar[r]  & \Omega_{\til{Y}}^1(\log(G))  \ar[r]
     & \cO_{\til{Y}}  \ar[r] & 0.
}
\end{equation}

If we let $\scr{E}_{S'}^{\vee}=\sigma_{\til{X}}^*\Omega^1_{\tilde{X}}(\log(\til{D}+S))\otimes_{\cO_{\til{Y}}}\cO_{S'}$, then the above sequence is the pull back via $\pi_{S'}$ via the following exact sequence on $S'$:
\begin{equation}\label{eq-ExtS}
\xymatrix{
0 \ar[r] & \sigma_{S}^*\Omega^1_{S}(\log(\Delta)) \ar@{=}[d] \ar[r] & \scr{E}_{S'}^\vee \ar@{=}[d] \ar[r] & \cO_{S'} \ar@{=}[d] \ar[r] & 0\\
0 \ar[r] & \Omega^1_{S'}(\log(B)) \ar[r] & \scr{E}_{S'}^\vee \ar[r] &  \cO_{S'} \ar[r] & 0.
}
\end{equation}
Moreover, then extension class of the exact sequence \eqref{eq-ExtS} is given by $\Phi(c_1(L))$ where $\Phi$ is the composition of the following natural homomorphism of cohomology groups ($\phi$ was given in \eqref{eq-injcoh}):
\begin{eqnarray*}
&&H^1(S, \cO_{S}^*)\stackrel{{\bf \delta}}{\longrightarrow} H^1(S, \Omega_{S}^1)\stackrel{\sigma_{S}^*}{\longrightarrow} H^1(S', \sigma_{S}^*\Omega_{S}^1)\\
&&\hskip 5cm \stackrel{\phi}{\longrightarrow} H^1(S', \sigma_{S}^*\Omega^1_{S}(\log(\Delta)))=H^1(S', \Omega^1_{S'}(\log(B))).
\end{eqnarray*}

\end{cor}

\subsubsection{Appendix: Orbifold structures on log pairs}\label{app-orbifold}

We follow \cite{GT16} (see also \cite{CKT16, Mum83}) to recall the definition of orbifold structures for general log pair. 

\begin{defn}[{\cite[Definition 2.3]{GT16}}]\label{defn-adapted}
Let $(S, \Delta=\sum_i\delta_i \Delta_i)$ be a log pair with $\delta_i=1-\frac{m_i}{n_i}$ where $m_i, n_i$ are integers satisfying $0\le m_i\le n_i$ and ${\rm gcd}(m_i, n_i)=1$.

A finite, Galois, flat and surjective morphism $f: S'\rightarrow (S, \Delta)$ is said to be adapted to $(S, \Delta)$ if the following conditions are satisfied:
\begin{enumerate}
\item
The variety $S'$ is a normal quasi-projective variety.
\item
$f^*\Delta_i$ is a Weil divisor, for every $i$. 
\item
The morphism $f$ is \'{e}tale at the generic point of ${\rm Supp}(\lfloor \Delta\rfloor)$.
\end{enumerate}

\end{defn}
\begin{defn}[{\cite[Definition 2.5]{GT16}}]
We say that a pair $(S, \Delta)$ has an orbifold structure at $x\in S$ if there is a Zariski open neighborbood $U_x\subset S$ of $x$ equipped with a morphism $f_x: V_x\rightarrow U_x$ adapted 
to $\left.(S, \Delta)\right|_{U_x}$. Furthermore, if $U_x$ is smooth and ${\rm Supp}(f_x^*(\Delta))$ is simple normal crossing, we say that the orbifold structure defined by $(U_x, f_x, V_x)$ is smooth.
\end{defn}

\begin{defn}[{\cite[Definition 2.7]{GT16}}]
Let $\cC=\{(U_\alpha, f_\alpha, S_\alpha)\}_{\alpha\in I}$ be a collection of ordered triples describing local orbifold structures on $S$. Let $\alpha,\beta\in I$ and define $S_{\alpha\beta}$
be the normalization of the fiber product $(S_\alpha\times_{U_\alpha\cap U_\beta}S_\beta)$ with the natural projection $g_{\alpha\beta}: S_{\alpha\beta}\rightarrow S_\beta$ and $g_{\beta\alpha}: 
S_{\alpha\beta}\rightarrow S_\beta$. We say that $\cC$ defines an orbifold structure on $S$ if $\bigcup_{\alpha\in I}U_\alpha=S$ and for each $\alpha,\beta\in I$, the two morphism $g_{\alpha\beta}$
and $g_{\beta\alpha}$ are \'{e}tale.
\end{defn}
Most constructions in standard algebraic geometry can be extended to the orbifold setting. These include the definitions of coherent orbifold (sub-)sheaves, Chern classes of orbifold sheaves, slope (semi-, poly-)stability of orbifold sheaves. Moreover, one can define orbifold tangent sheaf (resp. orbifold cotangent sheaf) for a given orbifold structure, which is denoted by $T^1_S(-\log\Delta)$ (resp. $\Omega^1_S(\log\Delta)$). For our limited purpose, we just need the log smooth case.

\begin{exmp}[{\cite[Example 2.8]{GT16}}]
Let $(S, \Delta)$ be a log smooth pair. There is a {\it canonical} orbifold structure defined as follows. For any $x\in S$, let $U_x$ be a 
Zariski neighborhood of $x$ where $\Delta_i|_{U_x}$ is given by the zero set of $f_i\in \cO_{U_x}$. Let $\{t_i\}^{k}_{i=1}$ parametrize each copy of $\bC$ in the produce $\bC^k\times U_x$. Then the subvariety $V_x\subset \bC^k\times U_x$ defined by the zero of $\{(t_i^{n_i}-f_i)\}^{k}_{i=1}$ admits a projection $\sigma_x$ onto $U_x$. The collection $\cC:=\{(U_x, \sigma_x, V_x)\}$ defines a {\it
smooth orbifold structure}. 


\end{exmp}

The following basic facts can be deduced from \cite[\S 3]{Mum83} and \cite{GT16}. For the definition of compatible orbifold sheaves, see \cite[Definition 3.1]{GT16}.

\begin{prop}[{see \cite[Lemma 3.5]{Mum83} and \cite[Proposition 3.3]{GT16}}]\label{prop-sameChern}
Let $(S, \Delta)$ be a log-smooth pair. Then the classes $c_1^2$ and $c_2$, as multilinear forms on $N^1(S)^{n-1}_{\bQ}$ (resp. $N^1(S)^{n-2}_{\bQ}$) are well-defined. Moreover, they are functorial under adapted morphisms.
\end{prop}

\begin{prop}[{\cite[3.1]{GT16}}]\label{prop-compatible}
Let $(S, \Delta)$ be a log-smooth pair. We can choose $\sigma_S$ in the previous subsection such that the orbifold structure defined $\sigma_S$ is compatible with the canonical orbifold structure, and the orbifold tangent sheaf of $(S, \Delta)$ with respect to $\sigma_S$ is compatible with its canonical orbifold tangent sheaf. As a consequence, $T_{S}(-\log(\Delta))$ is semistable with respect to 
$-(K_S+\Delta)$ if and only if $\sigma_{S}^*T_{S}(-\log(\Delta))=T_{S'}(-\log(B))$ is semistable with respect to $\sigma_{S}^*(-(K_S+\Delta))$.
\end{prop}
\begin{rem}\label{rem-compatible}
Behrouz Taji pointed out to me that, given two orbifold stuctures on a fixed pair $(S, \Delta)$, if we have the same ramification order along $\Delta$ then the corresponding orbifold-cotangent sheaves are compatible (see \cite[Proof of Theorem C]{GT16}).
\end{rem}





\section{Generalizations of Tian's semistability result}\label{sec-logTian}
\subsection{log smooth case}
\begin{thm}\label{thm-poly}
Assume that the log smooth Fano pair $(S, \Delta)$ is K-polystable. Then the orbifold tangent sheaf $T_S(-\log\Delta)$ is semistable with respect to $-(K_S+\Delta)$. 

Moreover, let $\scr{E}$ be the extension of the orbifold tangent sheaf $T_{S}(-\log \Delta)$ by $\cO_{S}$ with the extension class $\lambda \cdot c_1(-(K_S+\Delta))$ and $\lambda\in \bQ_{>0}$. Then $\scr{E}$ is slope semistable. 
\end{thm}

\begin{proof}
We will carry out the proof in several steps.

\noindent
{\bf Step 1:}
We carry out the construction in section \ref{sec-cone} by choosing $(\til{S}, \til{\Delta})=(S, \Delta)$ and a ramified covering $\sigma_S: S'\rightarrow S$ such that the orbifold structure defined by $\sigma_S$ is compatible with the canonical orbifold structure of the log smooth pair $(S, \Delta)$. Consider the pull back of the orbifold tangent sheaf with respect to $\sigma_S$, denoted by $\sigma_S^*T_S(-\log(\Delta))$ or by $T_{S'}(-\log(B))$, as in Definition \ref{defn-virSDel}. 
By Proposition \ref{prop-compatible}, we just need to show that the sheaf $\sigma_S^*T_S(-\log\Delta)$ is semistable with respect to $\sigma_S^*(-(K_S+\Delta))$. 

By the Yau-Tian-Donaldson conjecture for log smooth Fano pair proved in \cite{LTW17, TW18} \footnote{since we will be using approximation approach to deal with K-semistability in step 3, we just need the version involving uniform K-stability in \cite{LTW17}}, we know that there is a K\"{a}hler-Einstein metric $\omega$ on $(S, \Delta)$ in the sense that
\begin{enumerate}
\item $\omega$ satisfies the following equation:
\begin{equation}
Ric(\omega)=\omega+\sum_i \delta_i \{\Delta_i\}.
\end{equation}
\item
$\omega$ is smooth on $S\setminus {\rm Supp}(\Delta)$ and is quasi-isometric to the following model metric near $\Delta$:
\[
\sum_{k=p+1}^{n} \frac{\sqrt{-1}dz_k\wedge d\bar{z}_k}{|z_k|^{2\delta_i}}+\sum_{k=1}^p\sqrt{-1} dz_k\wedge d\bar{z}_k.
\]
\end{enumerate}
Pulling back $\omega$ by $\sigma_S: S'\rightarrow S$, we get a positive current $\omega'$, satisfying:
\begin{enumerate}
\item Outside ${\rm Supp}(B)$, $Ric(\omega')=\omega'$. Here $B=\sum_i (1-N+N\delta_i)\Delta'_i+(1-N)H'$ (see Notation \ref{nota-B})

\item $\omega'$ is smooth outside ${\rm Supp}(\Delta'+H')={\rm Supp}(B)$, and near ${\rm Supp}(B)$, $\omega'$ is quasi-isometric to the following model metric:
\begin{equation}\label{eq-modelsmooth}
\sum_{k=p+1}^n |w_k|^{-2d_k}\sqrt{-1} dw_k\wedge d\bar{w}_k+\sum_{k=1}^{p}\sqrt{-1} dw_k\wedge d\bar{w}_k,
\end{equation}
where $d_k=1-N+N\delta_k$ or $1-N$. Note that $d_k\le 0$ always holds.

\end{enumerate}
\bigskip

\noindent
{\bf Step 2: }
We use similar argument as \cite[pp. 22-23]{GT16}.
Let $\scr{F}$ be any coherent sheaf of $T_{S'}(-\log(B))(=\sigma_S^*T_{S}(-\log(\Delta))$ with rank $r=\rk(\scr{F})$.  Let $\scr{L}=(\wedge^r\scr{F})^{**}$. Then we get a holomorphic section $u$ of $\wedge^r(T_{S'}(-\log B))\otimes \scr{L}^{-1}$. Fix a smooth Hermitian metric $h_{\scr{L}}$ on $\scr{L}$. The metric $\omega':=\sigma_S^*\omega$ induces a Hermitian metric $h'$ on $\wedge^r(T_{S'}(-\log B))$. Because $\omega'$ is quasi-isometric to the model metric \eqref{eq-modelsmooth}, by using the local generator of $T_{S'}(-\log B)$ in Definition \ref{defn-virSDel}, it's easy to see that the metric $h'$ is bounded. Denote $|u|^2=|u|^2_{h'\otimes h_{\scr{L}}^{-1}}$. Then $|u|^2$ is a bounded function on $S'$ which is smooth on $S'\setminus B$.

To proceed, we need the following easy lemma.
 \begin{lem}
Let $E$ be a holomorphic vector bundle over a complex manifold $M$ with a smooth Hermitian metric $h$ and $u$ a holomorphic section of $E$. 
Let $F=F(\cdot)$ be a smooth concave function on $(0,+\infty)$ (i.e. $F''\le 0$), then we have the following inequality:
\begin{equation}\label{eq-LelongF}
\sddb F(|u|^2)\ge -F'(t)(R^E u,u)_h+(F''(t)t+F'(t))|\nabla u|^2,
\end{equation}
where $t=|u|_h^2$ and $R^E$ is the Chern curvature of $(E, h)$.

\end{lem}
\begin{proof}
We first claim the following holds. 
 For any $p\in M$, we can choose holomorphic coordinate chart $\{U_p, z_i\}$ centered at $p$ (i.e. $z_i(p)=0$ for all i) and holomorphic frames $\{s_\alpha\}_{1\le \alpha\le \rk(E)}$ over $\til{U}_p$ 
such that $h_{\alpha\bar{\beta}}=(s_\alpha, s_\beta)_{h}$ satisfies:
\[
h_{\alpha\bar{\beta}}(p)=\delta_{\alpha\beta}, \text{ and } \partial h_{\alpha\bar{\beta}} (p)=0.
\]
To see this, we first choose any holomorphic frame $\{\til{s}_\alpha\}$ of $E$ over a coordinate neighborhood $(\til{U}_p, \{z_i\})$ of $p$ such that the Hermitian metric $\til{h}_{\alpha\beta}=(\til{s}_\alpha, \til{s}_\beta)_h$ satisfies $\til{h}_{\alpha\bar{\beta}}(p)=\delta_{\alpha\beta}$.
Choose $s_\alpha=(\delta_{\alpha\beta}-\sum_i (\partial_{z_i}h_{\alpha\bar{\beta}}(p))z_i)\til{s}_\beta$. Then it's easy to verify that there exists $U_p\subset \til{U}_p$ such that $\{s_\alpha\}$ are holomorphic frames of $E$ over $U_p$ and satisfie the requirement.

Let $u=u_\alpha s_\alpha$ with $u_\alpha$ holomorphic over $U_p$. Then we can easily calculate that
$(\partial\bar{\partial}h_{\alpha\bar{\beta}})(p)=-(R^E s_\alpha, s_\beta)_h(p)$ and $\bar{\partial} |u|^2(p)=\left(u_\alpha \bar{\partial} \bar{u}_\alpha\right)(p)$ and
\begin{eqnarray*}
\partial\bar{\partial}|u|^2(p)&=&\left[(\partial u_\alpha)(\bar{\partial} \bar{u}_\alpha)+u_\alpha \bar{u}_\beta \partial\bar{\partial}h_{\alpha\bar{\beta}}\right](p)\\
&=&(\partial u_\alpha)(\bar{\partial}\bar{u}_\alpha)(p)-(R^E u, u)(p).
\end{eqnarray*}
Substituting these expression into $\partial\bar{\partial}F(|u|^2)$ and using Cauchy-Schwarz inequality, we easily get the inequality \eqref{eq-LelongF} since $p$ is arbitrary. 
\end{proof}
Applying the above lemma to $(M, E, h)=(S'\setminus B, \wedge^r(T_S')\otimes \scr{L}^{-1}, h')$ and $F(t)=\log(t+\tau^2)$ where $\tau>0$ is a constant we get the inequality
\begin{eqnarray}\label{eq-vectLelong}
\sddb\log(|u|^2+\tau^2)\ge \frac{|u|^2}{|u|^2+\tau^2}\left(R^{\scr{L}}-\frac{(R^{\wedge^rT_{S'}}u,u)_{h'\otimes h_{\scr{L}}^{-1}}}{|u|_{h'\otimes h_{\scr{L}}^{-1}}^2}\right),
\end{eqnarray}
where $R^{\scr{L}}$ is the Chern curvature of $(\scr{L}, h_{\scr{L}})$ and $R^{\wedge^r T_{S'}}$ is the Chern curvature of the Hermitian metric on $\wedge^r T_{S'}$ induced by $h'$.
In other words, for any $v=\frac{\partial}{\partial w_{m_1}}\wedge \cdots \wedge \frac{\partial}{\partial w_{m_r}}\in \wedge^r T_{S'}$, we have:
\begin{eqnarray*}
R^{\wedge^r T_{S'}}(v)&=&
R^{\wedge^r T_{S'}}\left(\frac{\partial}{\partial w_{m_1}}\wedge\cdots\wedge \frac{\partial}{\partial w_{m_r}}\right)\\
&=&\sum_{\alpha=1}^r \frac{\partial}{\partial w_{m_1}}\wedge \cdots \wedge \left(R^{T_{S'}}\frac{\partial}{\partial w_{m_\alpha}}\right) \wedge \cdots\wedge \frac{\partial}{\partial w_{m_r}}\\
&=:&(R^{T_{S'}})^{\wedge r}(v).
\end{eqnarray*}
Here $R^{T'_S}_{i\bar{j}}=R_{i\bar{j}k}^{'\;\;\;\;l}dw_k\otimes \frac{\partial}{\partial w_l}$ is the Riemannian tensor of the K\"{a}hler metric $\omega'$ on $S'\setminus {\rm Supp} B$ and so $g'^{i\bar{j}}R^{T_{S'}}_{i\bar{j}}=Ric(\omega')_{k}^l dw_k\otimes \frac{\partial}{\partial w_l}$.
As a consequence, 
\begin{eqnarray*}
\tr_{\omega'}\left(R^{\wedge^r T_{S'}}\right)&=&
g'^{i\bar{j}}R^{\wedge^r T_{S'}}_{i\bar{j}}=({\rm id}_{T_{S'}})^{\wedge r}
=r \cdot {\rm id}_{\wedge^r T_{S'}}.
\end{eqnarray*}
As in \cite[9]{CGP13}, \cite[p.2363]{CP16}, we can choose a family of cut-off function $\{\chi_\epsilon\}_{\epsilon>0}$ such that the $L^1$-norm of $\sddb \chi_\epsilon$ with respect to a smooth metric on $S'$ goes to zero as $\epsilon\rightarrow 0$. 
Wedging both sides of \eqref{eq-vectLelong} by $\chi_{\epsilon}\omega'^{n-1}_t$ and integrating on $S'$, we get by integration by parts:
\begin{equation}\label{eq-LePo}
-\int_{S'}\log(|u|^2+\tau^2)\sddb \chi_{\epsilon}\wedge \omega'^{n-1}\ge \int_{S'}\frac{|u|^2\chi_\epsilon}{|u|^2+\tau^2}\left(R^{\scr{L}}-\frac{(R^{\wedge^rT_S'}u, u)}{|u|^2}\right)\wedge \omega^{n-1}.
\end{equation}
Because $d_k\le 0$ and $|u|^2$ is bounded, it's easy to see that the left-hand-side goes to $0$ as $\epsilon\rightarrow 0$. 
The right-hand-side splits into two parts whose limits as $(\epsilon,\tau)\rightarrow (0,0)$ are given by (see \cite[p. 24]{GT16}):
\begin{eqnarray*}
I_1&=&\int_{S'}\frac{|u|^2\chi_\epsilon}{|u|^2+\tau^2}R^{\scr{L}}\wedge \omega'^{n-1} \xrightarrow{(\epsilon,\tau)\rightarrow (0,0)} c_1(\scr{L})\wedge [\omega']^{n-1}=\deg(\scr{F}) \\
I_2&=&-\int_{S'}\frac{|u|^2\chi_\epsilon}{|u|^2+\tau^2}\frac{1}{n}\frac{r |u|^2}{|u|^2}\omega'^n  \xrightarrow{(\epsilon,\tau)\rightarrow (0,0)}  - \frac{r}{n} [\omega']^n=-\frac{r }{n}\deg(T_{S'}(-\log(B))).
\end{eqnarray*}
So we get the wanted inequality:
\begin{eqnarray*}
\deg(\scr{F})\le \frac{\rk(\scr{F})}{n}\deg(T_{S'}(-\log(B))).
\end{eqnarray*}
\bigskip

\noindent
{\bf Step 3:} {\it Digression on extension of vector bundles}
\medskip

Let $E_1$ and $E_2$ be two holomorphic bundles over $S'$. Let $\psi\in \cA^{0,1}(End(E_1, E_2))$ be a $\bar{\partial}$-closed $\Hom(E_1,E_2)$-valued $(0,1)$-form. Then $\psi$ defines cohomology class $[\psi]$ in $H^{0,1}_{\bar{\partial}}(S', \cA(E_1^*\otimes E_2))\cong H^1(S', E_1^*\otimes E_2)$ which determines an extension, denoted by $\scr{E}:=\scr{E}([\psi])$, of $E_1$ by $E_2$:
\begin{equation}\label{eq-extension}
0\longrightarrow E_2\longrightarrow \scr{E}\longrightarrow E_1\longrightarrow 0.
\end{equation}
Choose Hermitian metrics $h_1$ on $E_1$ and $h_2$ on $E_2$. Denote by $D_1$ and $D_2$ the unique Chern connections associated to $h_1$ and $h_2$. Then the $(0,1)$-part of $D_1$ and $D_2$  give holomorphic structure on $E_1$ and $E_2$. Define a $\Hom(E_2, E_1)$-valued $(1,0)$-form $\bar{\psi}$ by:
\begin{equation}\label{eq-bvphi}
h_2(\psi(v), w)-h_1(v, \bar{\psi}(w))=0 \text{ for any } v\in E_1, w\in E_2.
\end{equation}
Consider the Hermitian metric on the complex vector bundle $E_1\oplus E_2$ given by $h:=h_1\oplus h_2$. Then the Chern connection associated to $h$ on the holomorphic vector bundle $\scr{E}$ is given by the following expression, whose 
$(0,1)$-part gives the holomorphic structure of $\scr{E}$:
\begin{equation}
D=
\left(\begin{array}{cc}
D_1& -\bar{\psi}\\
\psi & D_2
\end{array}
\right), \quad
D^{0,1}=
\left(\begin{array}{cc}
\bar{\partial}^{E_1}& 0\\
\psi & \bar{\partial}^{E_2}
\end{array}
\right).
\end{equation}

The extension class of the exact sequence \eqref{eq-extension} can also be given by the \v{C}ech cohomology as used in the proof of Proposition \ref{prop-base2cone}. We now explain how the extension sheaf $\scr{E}$ determines a holomorphic co-cycles $\phi_{\alpha\beta}\in 
End(E_1, E_2)(U_\alpha\cap U_\beta)$ which determines the extension class in $H^1(S', E_1^*\otimes E_2)$. 
First note that as complex vector bundles (without considering the holomorphic structure), $\scr{E}$ is isomorphic to $E_1\oplus E_2$.
If $v_\alpha=\{v_{\alpha,i}\}$ and $w_{\alpha}=\{w_{\alpha,r}\}$ are local holomorphic frames of $E_1$ and $E_2$ respectively, then 
we can assume that the holomorphic frames of $\cE$ are given by $\{v'_\alpha, w_\alpha\}$ such that:
\begin{eqnarray*}
\left
(v'_\alpha,  w_\alpha
\right)
=\left(
v_\alpha, w_\alpha
\right)
\left(\begin{array}{cc}
I_{E_1}& 0 \\
\til{\zeta}_\alpha & I_{E_2}
\end{array}
\right)
\end{eqnarray*}
where $\til{\zeta}_\alpha=((\til{\zeta})_i^r)$ is a $\rk(E_1)\times \rk(E_2)$ matrix-valued function which determines a homomorphism $E_1\rightarrow E_2$: $\zeta_\alpha(v_{\alpha,i})=(\til{\zeta}_\alpha)_i^r w_{\alpha,r}$. Moreover, because $v'_\alpha, w_\alpha$ are holomorphic frames, we see that the holomorphic structure of $\scr{E}$ is given by:
\begin{eqnarray*}
\bar{\partial}^{\scr{E}}=
\left(\begin{array}{cc}
\bar{\partial}^{E_1}& 0\\
\bar{\partial} \til{\zeta}_\alpha & \bar{\partial}^{E_2}
\end{array}
\right)
=:
\left(\begin{array}{cc}
\bar{\partial}^{E_1}& 0\\
\til{\psi}_\alpha & \bar{\partial}^{E_2}
\end{array}
\right)
\end{eqnarray*}

If $M^{E_i}_{\alpha\beta}$ are transition matrices between holomorphic frames of $E_i$ over $U_\alpha\cap U_\beta$, then the transition matrix between holomorphic frames of $\scr{E}$, defined by:
$(v'_{\beta}, w_{\beta})= (v'_\alpha, w_\alpha) M^E_{\alpha\beta}$, 
is then given by:
\[
M^E_{\alpha\beta}=
\left(\begin{array}{cc}
M^{E_1}_{\alpha\beta}& 0
\\
M^{E_2}_{\alpha\beta}\til{\zeta}_\beta -\til{\zeta}_\alpha M^{E_1}_{\alpha\beta} & M^{E_2}_{\alpha\beta}
\end{array}
\right)
=
\left(
\begin{array}{cc}
M^{E_1}_{\alpha\beta}& 0 \\
 \til{\phi}_{\alpha\beta} M^{E_1}_{\alpha\beta}& M^{E_2}_{\alpha\beta}
\end{array}
\right).
\]
where $\til{\phi}_{\alpha\beta}=(M^{E_2}_{\alpha\beta}) \til{\zeta}_\beta (M^{E_1}_{\alpha\beta})^{-1}-\til{\zeta}_\alpha$ is nothing but the matrix of $\zeta_\beta-\zeta_\alpha$ under the frames $\{v_{\alpha,i}\}$ and $\{w_{\alpha,r}\}$. Because $M^{E}_{\alpha\beta}$ is holomorphic on $U_\alpha\cap U_\beta$, we indeed have $\phi_{\alpha\beta}\in End(E_1, E_2)(U_\alpha\cap U_\beta)$.

Conversely starting from any $\phi=(\phi_{\alpha\beta})$, by using the partition of unity we can find a collection $\{\zeta_\alpha\}$ with $\zeta_\alpha \in \mathcal{A}(End(E_1, E_2))(U_\alpha)$ with $\phi_{\alpha\beta}=\zeta_\beta-\zeta_\alpha$. Since $\bar{\partial}\vphi_{\alpha\beta}=0$, we get a globally defined $End(E_1, E_2)$-valued $(0,1)$-form $\psi=\bar{\partial}\zeta_\alpha=\bar{\partial}\zeta_\beta$. 
Clearly, $\{\phi_{\alpha\beta}\}$ is identified with $\psi$ under the Dolbeaut isomorphism $H^1(S', E_1^*\otimes E_2)\cong H^{0,1}_{\bar{\partial}}(S', \cA(End(E_1, E_2))$.
We will use the equivalence of these two descriptions of the extension bundle implicitly in our discussion. See \cite[V.14]{Dem} for more discussions.

\bigskip

\noindent
{\bf Step 4:} {\it Proof of the second statement of Theorem \ref{thm-poly}} 
\medskip

By the discussion in Proposition \ref{prop-compatible}, we just need to show the following 

\begin{thm}\label{thm-ETstable}
Under the same assumption as Theorem \ref{thm-poly},
let $\scr{E}_{S'}$ be the extension of $T_{S'}(-\log B)$ by $\cO_{S'}$ with the extension class $\lambda \cdot c_1(T_{S'}(-\log(B)))$. Then $\scr{E}_{S'}$ is slope semistable
with respect to $\sigma_S^*(-(K_S+\Delta))$. 
\end{thm}

The curvature of $D$ is given by:
\begin{equation}
R=
\left(
\begin{array}{cc}
R^{E_1}-\bar{\psi}\wedge \psi & -D_1\circ\bar{\psi}-\bar{\psi}\circ D_2\\
\psi\circ D_1+D_2\circ \psi & R^{E_2}-\psi\wedge \bar{\psi}
\end{array}
\right)
=:\left(
\begin{array}{cc}
\cA & -\bar{\cB} \\
\cB & \cC
\end{array}
\right).
\end{equation}
In the following calculations, we will work on $S'\setminus {\rm Supp}(B)$ where $\sigma_{S}$ is \'{e}tale and $\omega'=\sigma_S^*\omega$ is a smooth K\"{a}hler-Einstein metric. For the simplicity of notations, we don't distinguish $\omega=\sqrt{-1}\sum_{i,j}g_{i\bar{j}}dz^i\wedge d\bar{z}^j$ on $S\setminus (\Delta\cup H)$ with $\omega'=\sqrt{-1}\sum_{i,j}g'_{i\bar{j}}dw_i\wedge d\bar{w}_j$ over $S'\setminus B$. We then have:
\begin{equation}
g^{i\bar{j}}R_{i\bar{j}}=
\left(
\begin{array}{cc}
g^{i\bar{j}}\cA_{i\bar{j}} & -g^{i\bar{j}}\bar{\cB}_{i\bar{j}}\\
g^{i\bar{j}}\cB_{i\bar{j}} & g^{i\bar{j}}\cC_{i\bar{j}}
\end{array}
\right)
=
\left(
\begin{array}{cc}
\tr_{\omega}\cA & -\tr_{\omega}\bar{\cB}\\
\tr_{\omega}\cB & \tr_{\omega}\cC
\end{array}
\right).
\end{equation}
It will be convenient to write the above data using local coordinate charts and holomorphic frames. Choose local coordinate $\{z^i\}_{1\le i\le n}$ and holomorphic frames $\{v_p\}_{1\le p\le \rk(E_1)}$ and $\{w_r\}_{1\le r\le \rk(E_2)}$. We can write $\psi\in \cA^{0,1}(End(E_1, E_2))$ as
\[
\psi(v_p)=\psi^{r}_p w_r\quad \text{ with }\quad \psi^{r}_p=\psi^{r}_{p \bar{j}} d\bar{z}^j.
\]
Then by \eqref{eq-bvphi} we have the following expression for $\bar{\psi}$:
\[
\bar{\psi}(w_s)=\bar{\psi}^{q}_s v_q,\quad  \bar{\psi}^{q}_s=(h_1)^{q\bar{p}}\overline{\psi^{r}_p} (h_2)_{rs}.
\]
We can calculate explicitly:
\begin{eqnarray*}
\cB(v_p)&=&(\psi\circ D_1+D_2\circ \psi)(v_p)= \psi((\theta_1)_p^q v_q)+D_2 (\psi_p^r w_r)\\
&=&-(\theta_1)^q_p\wedge\psi^r_q w_r+d\psi^r_p w_r-\psi^r_p (\theta_2)_r^s w_s\\
&=&\left(d\psi^r_p+(\theta_2)^r_s\psi^s_p+\psi^r_q\wedge (\theta_1)^{q}_p\right)w_r=:\cB_p^r w_r;
\end{eqnarray*}
and over $S'\setminus {\rm Supp}(B)$:
\begin{eqnarray}
\tr_{\omega}\cA&=&g^{i\bar{j}}(R^{E_1}_{i\bar{j}})_p^{q}- g^{i\bar{j}} (h_1)^{q\bar{\gamma}}\overline{\psi^{r}_{\gamma \bar{i}}}(h_2)_{rs}\psi^s_{p \bar{j}}\label{eq-trA}\\
\tr_{\omega}\cC&=&g^{i\bar{j}}(R^{E_2}_{i\bar{j}})^{s}_r + g^{i\bar{j}} \psi^{s}_{p \bar{j}} (h_1)^{p\bar{q}}\overline{\psi^{t}_{q \bar{i}}}(h_2)_{rt}. \label{eq-trC}
\end{eqnarray}

Now we specialize the above construction to the case where $(E_1, h_1)=(T_{S'}, g'=\sigma_{S}^*g)$ and $(E_2, h_2)=(\cO_{S'}, \frak{b})$ with $\frak{b}\in \bR_{>0}$, $\psi= \sigma_S^*\left(\frak{a} \cdot g_{i\bar{j}}dz^i\wedge d \bar{z}^j\right)=\frac{\frak{a}}{\sqrt{-1}}\cdot \sigma_S^*\omega \in H^{0,1}_{\bar{\partial}}(\Omega^1_{S'}(\log B))$ with $\fa\in \bC$.  
Then the following properties hold over $S'\setminus {\rm Supp}(B)$:
\begin{enumerate}
\item $(R^{E_1}_{i\bar{j}})_p^{q}=R_{i\bar{j}p}^{\;\;\;\;\;q} dz^i\wedge d\bar{z}$ so that $g^{i\bar{j}}(R^{E_1}_{i\bar{j}})_p^q=Ric(\omega)_p^q$.
\item $\cB=0$ and hence $\tr_{\omega}\cB=0$:
\begin{eqnarray*}
\cB_p^r&=& d\left( \frak{a} g_{p\bar{j}}d\bar{z}^j\right)+\frak{a} g_{q \bar{j}}d\bar{z}^j\wedge g^{q\bar{s}}\partial_{i}g_{p\bar{s}}dz^i\\
&=&\frak{a}\left(\partial_{i} g_{p\bar{j}}dz^i\wedge d\bar{z}^j+\partial_ig_{p\bar{j}}d\bar{z}^j\wedge dz^i\right)=0.
\end{eqnarray*} 
\item Substituting into \eqref{eq-trA}-\eqref{eq-trC}, we get the following identities over $S'\setminus {\rm Supp}(B)$:
\begin{eqnarray}
\tr_{\omega}\cA&=&Ric^{q}_p- g^{i\bar{j}} g^{q\bar{s}}\overline{\frak{a} g_{s\bar{i}}}\;\frak{b}\; \frak{a} g_{p\bar{j}}=Ric^q_p-\frak{b} |\frak{a}|^2 \delta^{q}_p\nonumber\\
&=&(1-\frak{b}|\frak{a}|^2)\delta^q_p. \label{eq-trA1}\\
\tr_{\omega}\cC&=&g^{i\bar{j}}\frak{a} g_{p\bar{j}}g^{p\bar{q}}\overline{\frak{a} g_{q\bar{i}}}\; \frak{b}=n \frak{b}|\frak{a}|^2.  \label{eq-trC1}
\end{eqnarray}
\end{enumerate}
By choosing $\frak{a}=\lambda \sqrt{-1}$ and $\frak{b}=\frac{1}{(n+1)\lambda^2}$, we get:
\begin{eqnarray}\label{eq-trAC2}
\tr_{\omega}\cA=\frac{n}{n+1} \delta^{q}_p, \quad
\tr_{\omega}\cC=\frac{n}{n+1}.
\end{eqnarray}
So we get:
\begin{equation}\label{eq-REcontr}
g^{i\bar{j}}R_{i\bar{j}}^{\scr{E}}=\frac{n}{n+1}{\rm id}_{\scr{E}}.
\end{equation}
We can now carry out similar argument as before. Let $\scr{F}$ be any subsheaf of $\scr{E}$ of rank $r$ and let $\scr{L}=\det(\scr{F})^{**}$. The injection $\scr{F}\rightarrow \scr{E}$ determines a nonzero section
$u$ of $\wedge^r\scr{E}\otimes \scr{L}^{-1}$. Denote by $h_{\scr{E}}$ the Hermitian metric $h'\oplus \frac{1}{(n+1)\lambda^2}$ on $\scr{E}$. Fix a smooth Hermitian metric $h_{\scr{L}}$ on $\scr{L}$. Then the point is again that $|u|^2_{h_{\scr{E}}\otimes h_{\scr{L}}^{-1}}$ is bounded.
The inequality \eqref{eq-LePo} becomes:
\begin{equation}\label{eq-LePo}
-\int_{S'}\log(|u|^2+\tau^2)\sddb \chi_{\epsilon}\wedge \omega'^{n-1}\ge \int_{S'}\frac{|u|^2\chi_\epsilon}{|u|^2+\tau^2}\left(R^{\scr{L}}-\frac{(R^{\wedge^r\scr{E}} u, u)}{|u|^2}\right)\wedge \omega'^{n-1}.
\end{equation}
Using \eqref{eq-REcontr} we get, 
\[
\tr_{\omega'}R^{\wedge^r\scr{E}}=\frac{n}{n+1} r\cdot {\rm id}_{\wedge^r\scr{E}}.
\]
As before, as $\epsilon\rightarrow 0$, the left-hand-side goes to 0. The right-hand-side decomposes into two parts with limits given by:
\begin{eqnarray*}
I_1&=&\int_{S'}\frac{|u|^2\chi_\epsilon}{|u|^2+\tau^2}R^{\scr{L}}\wedge \omega'^{n-1} \xrightarrow{(\epsilon,\tau)\rightarrow (0,0)} c_1(\scr{L})\wedge [\omega']^{n-1}=\deg(\scr{F}) \\
I_2&=&-\int_{S'}\frac{|u|^2\chi_\epsilon}{|u|^2+\tau^2}\frac{1}{n}\frac{n}{n+1}r\frac{(u,u)}{|u|^2}\omega'^n=-\frac{r}{n+1}\int_{S'}\frac{|u|^2\chi_\epsilon}{|u|^2+\tau^2}\omega'^n \\&& \hskip 1cm \xrightarrow{(\epsilon,\tau)\rightarrow (0,0)}  
- \frac{r}{n+1} [\omega']^n=-\frac{r}{n+1} \deg(\scr{E})=-\frac{r}{n+1}\deg(T_{S'}(-\log(B))) 
\end{eqnarray*}
So we get the wanted inequality:
\begin{eqnarray*}
\deg(\scr{F})\le  \frac{\rk(\scr{F})}{n+1}\deg(\scr{E}).
\end{eqnarray*}


\end{proof}
\begin{rem}
Similar to the smooth case, with the orbifold (or conical) Hermitian-Einstein metrics at hand,
one should be able to prove a stronger polystability result. Since we don't need it in this paper, we will be satisfied with the semistability result.
\end{rem}

\begin{thm}\label{thm-semi}
With the above notations, assume the log smooth Fano-pair $(S, \Delta)$ is K-semistable. Then the orbifold co-tangent sheaf $T_S(-\log \Delta)$ is slope semistable with respect to $-(K_S+\Delta)$.

Let $\scr{E}$ be the extension of the orbifold tangent sheaf $T_{S}(-\log \Delta)$ by the structure sheaf $\cO_{S}$ with the extension class $\lambda \cdot c_1(-(K_S+\Delta))$ and $\lambda\in\bQ_{+}$. Then $\scr{E}$ is slope semistable with respect to $-(K_S+\Delta)$. 
\end{thm}

To prove this theorem, by choosing an auxiliary very ample divisor $H$ we know that the log Fano pair $(S, \Delta_t):=(S,\Delta+\frac{1-t}{m}H)$ is K-stable for any $t\in (0,1)\cap \bQ$ such that Theorem \ref{thm-poly} applies. As $t\rightarrow 1$, the semistability inequality of $\Omega^1_{S}(\Delta_t)$ will give us the semistability inequality of $\Omega^1_{S}(\Delta)$. For this to work, we note that the orbifold structure of $(S, \Delta)$ is a orbifold sub-structure of $(S, \Delta_t)$ in the sense that the global adapted morphism for $(S, \Delta_t)$ also induces a global adapted morphism of $(S, \Delta)$.

\begin{proof}

Choose a sufficiently ample divisor $H\in |m(-K_S+\Delta)|$ for $m$ sufficiently divisible. Then by similar calculation as in \cite{LTW17}, we know that for any $t\in (0,1)\cap \bQ$, $(S, \Delta+\frac{1-t}{m}H)=:(S, \Delta_t)$ 
is K-polystable (actually it's uniformly K-stable). By Theorem \ref{thm-poly}, $T_S(-\log(\Delta_t))$ is semistable with respect to $-(K_S+\Delta)$. It's well known that a sheaf $\scr{E}$ is semistable if and only if its dual $\scr{E}^{\vee}$ is semistable. So we know that $\Omega^1_S(\log(\Delta_t))$ is semistable.

Now choose an adapted finite morphism $\sigma^t_S: (S', \Delta'_t)\rightarrow (S, \Delta_t)$ with $\Delta'_t=\Delta'+\frac{1-t}{m}H'$. We can assume that $\sigma^t_S: (S', \Delta')\rightarrow (S, \Delta)$
is an adapted finite morphism that is compatible with the canonical orbifold structure of $(S, \Delta)$ (see Remark \ref{rem-compatible}). Then we have a natural inclusion:
\[
(\sigma_S^t)^*\Omega^1_S(\log(\Delta))\longhookrightarrow (\sigma_S^t)^*\Omega^1_{S}(\log(\Delta_t)).
\]
Let $\scr{F}$ be any rank $r$ orbifold subsheaf of $\Omega^1_S(\log(\Delta))$. Then $\scr{F}':=(\sigma_S^t)^*\scr{F}$ is a subsheaf of $(\sigma_S^t)^*\Omega^1_S(\log(\Delta))$. By the above inclusion, $\scr{F}'$ is also a sub sheaf of $(\sigma_S^t)^*\Omega^1_{S}(\log(\Delta_t))$ which is semistable with respect to $(\sigma_S^t)^*(-(K_S+\Delta_t))$. 
So we get:
\begin{eqnarray*}
\frac{\deg(\scr{F})}{\rk(\scr{F})}&=& \frac{1}{\deg(\sigma_S^t)}\frac{\deg(\scr{F}')}{\rk(\scr{F}')}\le \frac{1}{\deg(\sigma_S^t)}\frac{(\sigma_S^t)^*((K_S+\Delta_t)\cdot \left(\sigma_S^t)^*(-(K_S+\Delta_t)\right)^{n-1}}{n}\\
&=&-t^n \frac{ (-(K_S+\Delta))^n}{n}=t^n\cdot \frac{\deg(\Omega^1_S(\log(\Delta))}{n}.
\end{eqnarray*}
By letting $t\rightarrow 1$, we see that $\Omega^1_S(\log(\Delta))$ is semistable. As a consequence, its dual $T_S(-\log(\Delta))$ is also semistable.

Let $\scr{E}_t^{\vee}$ be the extension sheaf of $\cO_{S'}$ by $(\sigma^t_S)^*(\Omega^1_S(\log(\Delta_t)))$. By Theorem \ref{thm-poly}, $\scr{E}_t^{\vee}$ is semistable. There is a natural map $\scr{E}^{\vee}\rightarrow \scr{E}_t^{\vee}$. Using the same argument as above, we get the second statement of Theorem \ref{thm-semi}. 

\end{proof}
\begin{rem}
With the help of the properness of log-Mabuchi energy established in \cite[Theorem 2.6]{LTW17} (following \cite{BBJ15}), it's easy to get that, under the assumption of Theorem \ref{thm-semi}, for any $t\in (0,1)$, there exists a conical K\"{a}hler metric $\omega_t\in 2\pi c_1(-(K_S+\Delta))$ satisfying:
\[
Ric(\omega_t)=t \omega_t+(1-t)\hat{\omega}_0,
\]
where $\hat{\omega}_0$ is a fixed conical K\"{a}hler metric on the smooth log pair $(S, \Delta)$ (see \cite{Li13}). One can also carry out the proof of Theorem \ref{thm-semi} by using such twisted conical K\"{a}hler-Einstein metrics similar to \cite{Tia92} and \cite[Proof of Theorem 4.1]{GT16}.
\end{rem}

\subsection{A result about singular log-Fano pairs}\label{sec-singularstable}

Let $(S, \Delta)$ be a log-Fano pair with klt singularities. Let $\mu_S: (\til{S}, \til{\Delta})\rightarrow (S, \Delta)$ be a log resolution such that $\til{\Delta}+\sum_i E_i$ is simple normal crossing. We can write:
\begin{equation}\label{eq-Ktrans2}
K_{\til{S}}+\til{\Delta}=\mu^*(K_S+\Delta)+\sum_j c_j E_j \text{ with } c_j>-1.
\end{equation}
The goal of this section is to prove the following technical result:
\begin{prop}\label{prop-singularstable}
Assume that $S$ is $\bQ$-factorial and that there exists a log resolution with $c_i\in (-1,0]$ for all $i$. Then the orbifold tangent sheaf $T_{\til{S}}(-\log(\til{\Delta})$ is slope semistable with respect to $\mu_S^*(-(K_S+\Delta)$.
\end{prop}
\begin{proof}
We just need to show $\Omega^1_{\til{S}}(\log(\til{\Delta}))$ is semistable with respect to $\mu_S^*(-(K_S+\Delta))$.

Choose $\theta_i\in (0,1)\cap \bQ$ such that $-\mu^*(K_S+\Delta)-\sum_i \theta_i E_i$ is ample. 
Choose $m$ sufficiently divisible and $H\in \left|-m (\mu_S^*(K_S+\Delta)-\sum_i\theta_i E_i)\right|$ be a very general smooth divisor such that $H+\til{\Delta}+E$ has simple normal crossings. Then by the same proof as \cite[Proof of Proposition 3.1]{LTW17}, we can show that if $m\gg 1$ then $(\til{S}, A_{(t,\epsilon)}))$ is K-stable where
$$
A_{(t,\epsilon)}= \frac{1-t}{m}H+\til{\Delta}+\sum_j \left((-c_j)+t\epsilon\theta_i+(1-t)\theta_i\right) E_j=:\frac{1-t}{m}H+\til{\Delta}+\sum_j \alpha_j E_j.
$$
Note that by assumption, $\alpha_j\in [0,1)$ for $0<1-t\ll 1$ and $0< \epsilon\ll 1$.
Let $\sigma_{\til{S}}^{(t, \epsilon)}: (S', A'_{(t, \epsilon)})\rightarrow (\til{S}, A_{(t, \epsilon)})$ be an adapted morphism which is compatible with the canonical orbifold structure of $(S, A_{(t, \epsilon)})$. Then $\sigma_{\til{S}}^{(t, \epsilon)}: (S', \Delta')\rightarrow (\til{S}, \til{\Delta})$ is compatible with the canonical orbifold structure of $(\til{S}, \til{\Delta})$ and there is a natural inclusion
\[
(\sigma^{(t,\epsilon)})^*\Omega^1_{\til{S}}(\log(\til{\Delta}))\hookrightarrow (\sigma^{(t,\epsilon)})^*\Omega^1_{\til{S}}(\log(A_{(t,\epsilon)})).
\]
By Theorem \ref{thm-semi}, $(\sigma^{(t, \epsilon)}_{\til{S}})^*\Omega^1_{\til{S}}(\log(A_{(t, \epsilon)}))$ is semistable with respect to $(\sigma_{\til{S}}^{(t, \epsilon)})^*(-(K_{\til{S}}+A_{(t,\epsilon)}))$.  

For any rank $r$ sub sheaf $\scr{F}$ of $\Omega^1_{\til{S}}(\log(\til{\Delta}))$, $\scr{F}'=(\sigma^{(t, \epsilon)}_{\til{S}})^*\scr{F}$ is a subsheaf of $(\sigma^{(t, \epsilon)}_{\til{S}})^*\Omega^1_{\til{S}}(\log(\til{\Delta}))$. \begin{eqnarray*}
\frac{\deg(\scr{F})}{\rk(\scr{F})}&=&\frac{1}{\deg(\sigma_{\til{S}}^{(t, \epsilon)})}\frac{\deg(\scr{F}')}{\rk(\scr{F}')}\\
&\le&\frac{1}{\deg(\sigma_{\til{S}}^{(t, \epsilon)})} \frac{((\sigma^{(t,\epsilon)}_{\til{S}})^*((K_{\til{S}}+A_{(t, \epsilon)}))\cdot  ((\sigma_{\til{S}}^{(t,\epsilon)})^*(-(K_{\til{S}}+A_{(t, \epsilon)})))^{n-1}}{n}\\
&=&-t^n\frac{\left(\mu_S^*(-(K_S+\Delta))-\epsilon\sum_i\theta_iE_i\right)^n}{n}.
\end{eqnarray*}
Letting $(t,\epsilon)\rightarrow (1,0)$ and noticing that:
\[
\deg(\Omega^1_{\til{S}}(\log(\til{\Delta})))=(K_{\til{S}}+\til{\Delta})\cdot (\mu_S^*(-(K_S+\Delta)))^{n-1}=-(\mu_S^*(-(K_S+\Delta)))^n
\]
we get the wanted inequality:
\begin{eqnarray*}
\frac{\deg(\scr{F})}{\rk(\scr{F})}\le \frac{\deg(\Omega^1_{\til{S}}(\log(\til{\Delta})))}{n},
\end{eqnarray*}
which implies $\Omega^1_{\til{S}}(\log(\til{\Delta}))$ is semistable with respect to $\mu_S^*(-(K_S+\Delta))$.

\end{proof}

\subsection{Log Calabi-Yau case}\label{sec-logCY}

In this section, we sketch the proof of Theorem \ref{thm-mainlogCY} and leave the details to the reader, since it is very similar to the proof of Theorem \ref{thm-poly}. First it is known that there is a K\"{a}hler Ricci-flat metric $\omega_{\CY}$ in the cohomology class $2\pi c_1(L)$ with on the log smooth pair $(S, \Delta)$. By \cite[Theorem 6.3]{GP16} $\omega_{\CY}$ has cone singularities along $\Delta_i$ if $\delta_i\in (0,1)$ and cusp singularities 
if $\delta_i=1$. In other words, $\omega_\CY$ is locally quasi-isometric to the model metric:
\begin{equation*}
\sum_{k=p+1}^m \frac{\sqrt{-1} dz_k\wedge d\bar{z}_k}{|z_k|^{2\delta_k}}+\sum_{k=m+1}^{n} \frac{\sqrt{-1} dz_k\wedge d\bar{z}_k}{|z_k|^2(-\log|z_k|^2)}+\sum_{k=1}^p\sqrt{-1}dz_k\wedge d\bar{z}_k.
\end{equation*}

 Use the same notations as proof of Theorem \ref{thm-poly}, we let $g$ be the metric associated to $\omega_{\CY}$ and $g'$ the pull back of $g$ by the Galois covering $\sigma_S: S'\rightarrow S$. Then $\sigma_S^*\omega$ is locally quasi-isometric to:
\begin{equation*}
\sum_{k=p'+1}^m w_k^{-2d_k}dw_k\wedge d\bar{w}_k+\sum_{k=m+1}^n\frac{\sqrt{-1}dw_k\wedge d\bar{w}_k}{|w_k|^2(-\log|w_k|^2)}+\sum_{k=1}^{p'}\sqrt{-1} dw_k\wedge d\bar{w}_k.
\end{equation*}
Since the section $u$ of $\wedge^r (T_{S'}(-\log B))\otimes (\wedge^r\scr{F})^{-1}$ associated to any rank $r$ subsheaf $\scr{F}$ of $T_{S'}(-\log B))$ still has a bounded norm on $S'$, the first statement as in proof of Theorem \ref{thm-poly} can be proved in similar way as before following the argument in \cite[p.23]{GT16}. 

For the second statement, we let $(E_1, h_1)=(T_{S'}, g')$ and $(E_2, h_2)=(\cO_{S'}, \frak{b})$, $\psi=g'$, $\frak{a}=1$ and $\frak{b}$ any positive constant.
Then \eqref{eq-trA1}-\eqref{eq-trC1} becomes:
\begin{equation}
\tr_{\omega'}\cA=-\frak{b}\cdot \delta^{\beta}_\alpha, \quad \tr_{\omega'}\cC=n \frak{b}.
\end{equation}
In particular 
\[
\tr_{\omega'}R^{\scr{E}}=-\frak{b}\cdot {\rm id}_{\scr{E}}+(0\oplus (n+1)\frak{b})=:-\frak{b}\cdot \id_{\scr{E}}+\eta.
\]
Then we have:
\[
\tr_{\omega'}R^{\wedge^r\scr{E}}=(-\frak{b}\cdot {\rm id})^{\wedge^r}+\eta^{\wedge^r}=-\frak{b} \cdot r\cdot \id_{\wedge^r\scr{E}}+\eta^{\wedge^r}.
\]
Note that $\eta^{\wedge^r}=0$ if $r>1$ and in general we always have:
\[
\frac{(\eta^{\wedge^r}u, u)}{|u|^2}\le \lambda_{\max}(\eta^{\wedge^r})\le (n+1)\frak{b}.
\]
As in \eqref{eq-LePo}, we have the following inequality:
\begin{equation}\label{eq-LePo2}
-\int_{S'}\log(|u|^2+\tau^2)\sddb \chi_{\epsilon}\wedge \omega'^{n-1}\ge \int_{S'}\frac{|u|^2\chi_\epsilon}{|u|^2+\tau^2}\left(R^{\scr{L}}-\frac{(R^{\wedge^r\scr{E}} u, u)}{|u|^2}\right)\wedge \omega'^{n-1}.
\end{equation}
As $\epsilon\rightarrow 0$, the left-hand-side goes to 0. The right-hand-side decomposes into three parts with estimates:

\begin{eqnarray*}
I_1&=&\int_{S'}\frac{|u|^2\chi_\epsilon}{|u|^2+\tau^2}R^{\scr{L}}\wedge \omega'^{n-1} \xrightarrow{(\epsilon,\tau)\rightarrow (0,0)} c_1(\scr{L})\wedge [\omega']^{n-1}=\deg(\scr{F}) \\
I_2&=&-\int_{S'}\frac{|u|^2\chi_\epsilon}{|u|^2+\tau^2}\frac{1}{n} (-\frak{b}\cdot r)\frac{(u,u)}{|u|^2}\omega'^n=\frac{r\cdot \frak{b}}{n}\int_{S'}\frac{|u|^2\chi_\epsilon}{|u|^2+\tau^2}\omega'^n \\&& \hskip 1cm \xrightarrow{(\epsilon,\tau)\rightarrow (0,0)}  
\frac{r\cdot \frak{b}}{n} [\omega']^n=\frac{r\cdot \frak{b}}{n}(\sigma_S^*L)^{\cdot n}\\
I_3&=&-\int_{S'}\frac{|u|^2\chi_\epsilon}{|u|^2+\tau^2}\frac{(\eta^{\wedge r}u, u)}{|u|^2}\frac{1}{n}\omega'^n\ge -\frac{n+1}{n}\frak{b} \int_{S'}\omega'^n.
\end{eqnarray*} 
So we get the inequality:
\begin{eqnarray*}
\deg(\scr{F})\le \frak{b} \frac{\rk(\scr{F})}{n}(\sigma_S^*L)^{\cdot n}+\frac{n+1}{n}\frak{b}(\sigma_S^*L)^{\cdot n}.
\end{eqnarray*}
Note that $\deg(\scr{E})=\deg(T^1_{S'}(-\log(B))=0$. By letting $\frak{b} \rightarrow 0$, we get the wanted inequality: $\deg(\scr{F})/\rk(\scr{F})\le 0$.

\section{Applications}

\subsection{Local Euler numbers for 2-dimensional log canonical cones}\label{sec-volEuler}

 Let $(X, D, x)$ be a log terminal singularity and let $\Val_{X,x}$ denote the space of real valuations on $\cO_X$ whose center is at $x$. For any $v\in \Val_{X,x}$, denote by $A_{(X,D)}(v)$ its 
 log discrepancy (see \cite{JM12, BFFU15}) and by $\vol(v)$ its volume (see \cite{ELS03}). Then we recall:
 \begin{defn}[see \cite{Li15, LL16}]\label{defn-hvol}
The normalized volume of a log terminal singularity $(X,D,x)$ is defined to be:
\begin{equation}\label{eq-defhvol}
\hvol(x, X, D):=\inf_{v\in \Val_{X,x}} A_{(X,D)}(v)^{n} \vol(v). 
\end{equation}
\end{defn}
It was proved in \cite{Li15} that $\hvol(x, X, D)>0$ if $(X, D, x)$ is log terminal. H. Blum \cite{Blu16} proved that the infimum in \eqref{eq-defhvol} is actually obtained. 
The normalized volume of cone singularities over K-semistable log pairs can be calculated exactly:
\begin{thm}[\cite{Li15b,LL16,LX16}]\label{thm-semivol}
Let $(S, \Delta)$ be a log-Fano pair and $L$ an ample $\bQ$-Cartier divisor such that $-(K_S+\Delta)=\lambda \cdot L$ for $\lambda\in \bQ_{>0}$. Let $X=C(S, L)$ be the corresponding orbifold affine cone
and $D$ the divisor on $X$ corresponding to $\Delta$. Then $(S, \Delta)$ is K-semistable if and only if  $\hvol(x, X,D)=\lambda^{n+1} L^n=\lambda (-(K_S+\Delta))^n$.
\end{thm}
We need a more general result which deals with the case when a klt singularity degenerates to a K-semistable cone.
\begin{thm}[\cite{LX16}]\label{thm-degvol}
Let $(X, D, x)$ be a klt singularity and $v=\ord_S \in \Val_{X,x}$ be a divisorial valuation whose associated graded ring is the coordinate ring of a log Fano cone singularity $(X_0, D_0, x_0)$ ($S$ is called a Koll\'ar component in \cite{LX16}). Assume $(X_0, D_0, x_0)$ is K-semistable. Then we have the identity: 
\begin{equation}
\hvol(x,X,D)=\hvol(x_0, X_0, D_0)=A_{X,D}(S)^n\cdot \vol(\ord_S).
\end{equation}

\end{thm}

\begin{proof}[Proof of Proposition \ref{prop-conj}]
Assume $(X,D)$ be an orbifold cone over $(\bP^1,\Delta=\sum_i \delta_i p_i)$ with the orbifold line bundle denoted by $L$. Choose $k$ sufficiently divisor such that $k L$ is genuine line bundle. 
Denote by $(Z, D_Z, z)$ the ordinary affine cone over $(\bP^1, \Delta)$ with the polarization $kL$. Then we get a degree $k$ map $\sigma: (X, D, x)\rightarrow (Z, D_Z, z)$ with $\sigma^*(K_{Z}+D_Z)=K_X+D$. Because $(S, \Delta)$ is K-semistable, by the above theorem we have, for $n=1$:
\[
\hvol(x, X, D)=\lambda^{n+1}L^n=k\cdot (\lambda k^{-1})^{n+1}(kL)^n =k \cdot \hvol(z, Z, D_Z).
\]
By Lemma \ref{lem-Eulerdeg}, $e_{\orb}(x, X, D)=k\cdot e_{\orb}(z,Z,D_Z)$. If the conjecture holds for $(Z,D_Z,z)$, then it holds for $(X, D,x)$. So we can assume $L$ is a genuine line bundle. 

Now we apply the construction in section \ref{sec-cone} to $(S, \Delta)=(\bP^1, \sum_i \delta_i p_i)$. Then $\til{X}$ is just the blow-up of $x\in X$ and $(\til{S}, \til{\Delta})=(S, \Delta)$.
Let $\sigma_S: S'\rightarrow S$ be a branched covering of degree $N$ such that $\sigma_S^*\Delta$ is a Weil divisor. 

By Corollary \ref{cor-logsmooth}, $\sigma_{\til{X}}^*\Omega^1_{\til{X}}(\log(\til{D}+S))$ is equal to $\pi_{S'}^*\scr{E}$ where $\scr{E}$ is the extension of $\cO_{S'}$ by $\Omega^1_{S'}(\log(B))$ with the extension class given by $c_1(\sigma_S^*L)$. Because $(S, \Delta)$ is K-semistable, by 
Theorem \ref{thm-semi}, $\Omega^1_{S'}(\log(B))$ is slope semistable. So by  Definition \ref{defn-locEuler} and Theorem \ref{thm-coneuler}, we have
\[
e_{\orb}(x, X, D)=-\frac{c_2(\mu_Y, \sigma_{\til{X}}^*\Omega^1_{\til{X}}(\log(\til{D}+E_x))}{N}=\frac{c_1(\scr{E})^2}{4 N \deg(\sigma_S^*L)}.
\]
Note that $\deg_{S'}(\sigma_S^*L)=N \deg_S(L)$ and
\begin{eqnarray*}
c_1(\scr{E})&=&\int_{S'} c_1(\scr{E})=\int_{S'} c_1(\Omega^1_{S'}(\log(B)))=\int_{S'} c_1(K_{S'}+B)\\
&=&\int_{S'}c_1( \sigma_S^*(K_S+\Delta))=-\lambda \cdot \deg_{S'}(\sigma_S^*L).
\end{eqnarray*}
So we easily get the wanted identity:
\[
e_{\orb}(x, X, D)=\frac{\lambda^2 N^2 \deg_S(L)^2}{4N^2 \deg_S(L)}=\frac{\lambda^2 \deg_S(L)}{4}=\frac{\hvol(x, X, D)}{4}.
\]
\end{proof}
 
\begin{proof}[Proof of Proposition \ref{prop-conjCY}]
This is proved in the same way as Proposition \ref{prop-conj} by replacing Theorem \ref{thm-semi} by Theorem \ref{thm-mainlogCY} and noticing that $c_1(\cE)=c_1(-(K_S+\Delta))=0$.
\end{proof}
 
Now we specialize to the case $(X, D, x)=(\bC^2, \sum_{i=1}^m \delta_i L_i, 0)$ where $L_i=\{b_i z_1-a_i z_2=0\}$ are lines passing through $0\in \bC^2$. Then the natural $\bC^*$-action on $\bC^2$
makes $(X,D, x)$ an affine cone over $(\bP^1, \sum_i \delta_i p_i)$ with $p_i=[a_i, b_i]\in\bP^1$. Without the loss of generality, we assume $0\le \delta_1\le \delta_2\le\dots\le \delta_m\le 1$ and denote $\delta=\sum_{i=1}^m \delta_i$ and $\delta'=\sum_{i=1}^{m-1}\delta_i=\delta-\delta_m$.
We have the following cases:
\begin{enumerate}
\item $\delta>2$. $(\bC^2, D)$ is not log canonical. Then by \cite[Theorem 8.7]{Lan03}, $e_{\orb}(0; \bC^2, D)=0$. 
\item $\delta=2$. This is the log-Calabi-Yau case. By Proposition \ref{prop-conjCY}, we get $e_{\orb}(0; \bC^2, D)=0$.
\item $\delta<2$ and $\delta_m\ge \delta'$, by \cite[Theorem 8.7]{Lan03}, $e_{\orb}(0; \bC^2, D)=(1-\delta+\delta_m)(1-\delta_m)$ which is $0$ if $\delta_m=1$ (log canonical case). 

If $\delta_m<1$, then $(\bC^2, D)$ is klt and is unstable with respect to the natural rescaling vector field.
Without loss of generality, we can assume $p_m=\{0\}\in \bP^1=\bC\cup \{\infty\}$. There is then a $\bC^*$-equivariant degeneration of $(\bC^2, D)$ to the log-Fano cone $(\bC^2, \sum_i D'_i=\delta_m \{0\}+\delta'\{\infty\})$ with the $\bC^*$-action generated by $(1-\delta') z_1\frac{\partial}{\partial z_1}+(1-\delta_m) z_2\frac{\partial}{\partial z_2}$.
This corresponds to the jumping of metric tangent cone as explained in \cite[p.34]{BS17}.

 It's easy to check that the quotient of $(\bC^2, D')$ is given by $(\bP^1, \gamma(\{0\}+\{\infty\}))$ where $\gamma$ is determined by the following identities (see Example \ref{exmp-C2D}):
\[
\frac{1-\delta'}{1-\delta_m}=\frac{b}{a},  \quad a,b\in \mathbb{N}, \quad {\rm gcd}(a,b)=1, \quad \gamma=1-\frac{1-\delta'}{b}=1-\frac{1-\delta_m}{a}.
\]
Since $(\bC^2, \gamma (p_m+\bar{p}_m))$ is K-semistable and hence the log Fano cone $(\bC^2, D')$ is also K-semistable, by Theorem \ref{thm-degvol} we know that 
(use Theorem \ref{thm-semivol} with $\lambda=b(1-\delta_m)+a(1-\delta')$)
$$\hvol(0; \bC^2, D)/4=\hvol(0; \bC^2, D')/4=(1-\delta')(1-\delta_m)=(1-\delta+\delta_m)(1-\delta_m).$$

\item $\delta<2$ and $\delta_m< \delta'$. Then $(\bP^1, \sum_i \delta_i p_i)$ is K-stable (called the ``stable regime" in \cite{BS17}) . So by Proposition \ref{prop-conj}, $e_{\orb}(0; \bC^2, D)=\hvol(0; \bC^2, D)/4=(2-\delta)^2/4$.  
\end{enumerate}
\begin{proof}[Proof of Corollary \ref{cor-conetrue}]
$(X, D, x)$ is an orbifold cone over $(\bP^1, \sum_i \delta_i p_i)$ with the orbifold line bundle $L$. By choosing $k$ sufficiently divisible, $kL$ is a genuine line bundle. Let $(Z, D_Z, z)$ be the affine cone over $(\bP^1, \sum_i \delta_i p_i)$ with polarization $kL$. On the other hand $kL=d \cO_{\bP^1}(1)$ for some $d\in \bZ_{>0}$ and there is a Galois covering of degree $d$: $(\bC^2, \sum_i \delta_i L_i, 0)\rightarrow (Z, D_Z, z)$ where $L_i$ are lines given by $p_i\in \bP^1$. By Lemma \ref{lem-Eulerdeg}, we then have:
\[
e_{\orb}(x, X, D)=k\cdot e_{\orb}(z, Z, D_Z)=\frac{k}{d} e_{\orb}(0, \bC^2, \sum_i \delta_i L_i).
\]
By Theorem \ref{thm-semivol} and Theorem \ref{thm-degvol}, we have the same relation for $\hvol/4$ (see also \cite[Theorem 1.7]{LX17}). So the statement follows from the discussion above for the case of $(0, \bC^2, \sum_i \delta_i L_i)$.
\end{proof}




\subsection{Logarithmic Miyaoka-Yau inequalities for K-semistable pairs}\label{sec-Chineq}

In this section, we give the proof of Theorem \ref{thm-Chineq} and Theorem \ref{thm-ChineqCY}. 
First recall the well-known Bogomolov-Gieseker inequality:
\begin{thm}[{see \cite[4]{Miy87}}]\label{thm-Bog}
Let $S'$ be a projective manifold and let $H$ be a nef line bundle on 
$S'$. If $E$ is any reflexive coherent sheaf of rank $r$ that is stable with respect to $H$, then $E$ verifies:
\begin{equation}
\Delta(E)\cdot H^{n-2}\ge 0,
\end{equation}
where $\Delta(E)$ is the Bogomolov discriminant:
\[
\Delta(E):=2 r c_2(E)-(r-1) c_1(E)^2.
\]
\end{thm}
Now we let $\scr{E}$ to be the extension of orbifold tangent sheaf $T_S(-\log(\Delta))$ by $\cO_S$ with the extension class $c_1(-(K_S+\Delta))$. Let $\sigma_S: S'\rightarrow S$ be the ramified covering as in the commutative diagram \eqref{eq-CDlogsmooth}. Then $\sigma_S^*\scr{E}=\scr{E}_{S'}$ where
$\scr{E}_{S'}$ is the extension of $\sigma_S^*T_S(-\log(\Delta))=T_{S'}(-\log(B))$ by $\cO_{S'}$ with the extension class $\sigma_S^*c_1(-(K_S+\Delta))=c_1(-(K_{S'}+B))$ as in Theorem \ref{thm-ETstable}. 
By Theorem \ref{thm-semi},
$\scr{E}_{S'}$ is slope semistable with respect to $-(K_{S'}+B)$. On the other hand, we get (cf. \cite[pp.29]{GT16}):
\[
\Delta(\scr{E}_{S'})=2 (n+1) \sigma_S^*(c_2(S, \Delta)) - n \sigma_S^*(c_1(S, \Delta))^2.
\]
Based on the fact in Proposition \ref{prop-sameChern}, Theorem \ref{thm-Chineq} follows immediately from Theorem \ref{thm-semi} by applying Theorem \ref{thm-Bog} to $\scr{E}_{S'}$. 

Theorem \ref{thm-ChineqCY} follows the same argument by replacing Theorem \ref{thm-main} by Theorem \ref{thm-mainlogCY}, and noticing that $c_1(S, \Delta)=0$ if $(S, \Delta)$ is log-Calabi-Yau.

\begin{rem}\label{rem-ChineqCY}
We end this paper by making some general remarks of the above proof of the Miyaoka-Yau inequalities. Firstly one can weaken the log-smooth assumption under suitable situations. For example one can replace the log smooth assumption by the conditions (i) or (ii) in \cite[Theorem B]{GT16} at least in the log-Calabi-Yau case. So one sees that the advantage of the above proofs is that we don't need the detailed information of the curvatures of the singular K\"{a}hler-Einstein metrics. The disadvantage however is that the equality case is not immediately clear. However methods used \cite{GKPT15} for characterizing the identity case of Miyaoka-Yau inequalities on singular canonically polarized varieties might also be useful for studying the identity case in the singular Fano/Calabi-Yau case. 

On the other hand, if one tries to prove the Miyaoka-Yau type inequality directly using K\"{a}hler-Einstein metrics as in Yau's proof, one needs enough regularity of the singular K\"{a}hler-Einstein 
metrics to identify the correction to the $L^2$-norm of the traceless Riemannian curvature associated to any singular point, which is in general quite difficult at present for general log canonical pairs.
In the case when $(S, \Delta)$ is log smooth with irreducible $\Delta$, Song-Wang \cite{SW12} used the regularity results (e.g. polyhomogeneity) from \cite{JMR16}. More generally when $\Delta$ is simple normal crossing, the polyhomogeneity property for K\"{a}hler-Einstein metrics on $(S, \Delta)$ was announced by Rubinstein-Mazzeo. 
In the case of log canonical surfaces, Borbon-Spotti conjectured in \cite{BS17} that the correction term associated to any point is precisely one less than the volume density of the K\"{a}hler-Einstein metric and, as mentioned in the introduction, that the volume densities should match Langer's local Euler numbers (at least for log terminal surface singularities). The main part \cite{BS17} is to study the behavior of K\"{a}hler-Einstein metrics near the singularities when the boundary divisors have good configurations (more precisely when the metric cone at any point is isomorphic to the germ of the point itself).

\end{rem}

\vspace{9mm}

\noindent
Department of Mathematics, Purdue University, West Lafayette, IN 47907-2067

\noindent
{\it E-mail address:} li2285@purdue.edu

\end{document}